\documentclass[11pt,reqno]{amsart} 

\usepackage{amsmath,amssymb,amsthm,blkarray}
\usepackage[mathscr]{euscript}
\usepackage{color}
\usepackage{hyperref,doi}
\usepackage{epsf}
\usepackage{enumitem}
\usepackage{graphicx}
\usepackage{array, multirow}

\usepackage[margin=1in]{geometry}

\newtheorem{theorem}{Theorem}[section]
\newtheorem{prop}[theorem]{Proposition}

\newtheorem{lemma}[theorem]{Lemma}
\newtheorem{corollary}[theorem]{Corollary}

\theoremstyle{definition}
\newtheorem{definition}[theorem]{Definition}

\newtheorem{example}[theorem]{Example}

\newtheoremstyle{underline}
{-1.5mm}        
{}              
{\parshape 1 1em \dimexpr\textwidth-1em} 
{}              
{\bfseries}              
{:}             
{1.5mm}         
{{\underline{\thmname{#1}\thmnumber{ #2}}}}  

\newtheoremstyle{underline2}
{-1.5mm}{}{\parshape 1 1.5em \dimexpr\textwidth-1.5em}{}{\bfseries}{:}{1.5mm}
{{\underline{\thmname{#1}\thmnumber{ #2}}}}

\theoremstyle{underline}
\newtheorem{case}{Case}
\theoremstyle{underline2}
\newtheorem{subcase}{Subcase}
\numberwithin{subcase}{case}

\newcommand{\conv}[1]{\mathrm{conv}\left\{#1\right\}}

\newcommand{\R}{\mathbb{R}}
\newcommand{\Z}{\mathbb{Z}}
\newcommand{\N}{\mathbb{N}}

\newcommand{\Pq}{\Delta_{(1,\bq)}}

\newcommand{\ba}{\mathbf{a}}
\newcommand{\bb}{\mathbf{b}}
\newcommand{\be}{\mathbf{e}}
\newcommand{\bbf}{\mathbf{f}^{\pi}}
\newcommand{\bbg}{\mathbf{g}^{\pi}}

\newcommand{\bq}{\mathbf{q}}
\newcommand{\br}{\mathbf{r}}

\newcommand{\bt}{\mathbf{t}}

\newcommand{\bx}{\mathbf{x}}
\newcommand{\by}{\mathbf{y}}
\newcommand{\sQ}{\mathcal{Q}}
\newcommand{\sG}{\mathcal{G}}
\newcommand{\sA}{\mathcal{A}}
\newcommand{\sM}{\mathcal{M}}
\newcommand{\sR}{\mathcal{R}}

\newcommand{\inG}{in_{<_{lex}}(\sG)}

\newcommand{\ds}{\displaystyle}
\newcommand{\zsupp}[1]{\mathrm{supp}_{\mathbf{z}}\left(#1\right)}
\newcommand{\abs}[1]{ \left\lvert#1\right\rvert} 

\newcommand\commentout[1]{}

\makeatletter
\def\@tocline#1#2#3#4#5#6#7{\relax
  \ifnum #1>\c@tocdepth 
  \else
    \par \addpenalty\@secpenalty\addvspace{#2}%
    \begingroup \hyphenpenalty\@M
    \@ifempty{#4}{%
      \@tempdima\csname r@tocindent\number#1\endcsname\relax
    }{%
      \@tempdima#4\relax
    }%
    \parindent\z@ \leftskip#3\relax \advance\leftskip\@tempdima\relax
    \rightskip\@pnumwidth plus4em \parfillskip-\@pnumwidth
    #5\leavevmode\hskip-\@tempdima
      \ifcase #1
       \or\or \hskip 1em \or \hskip 2em \else \hskip 3em \fi%
      #6\nobreak\relax
    \hfill\hbox to\@pnumwidth{\@tocpagenum{#7}}\par
    \nobreak
    \endgroup
  \fi}
\makeatother

\makeatletter
\@namedef{subjclassname@2020}{%
  \textup{2020} Mathematics Subject Classification}
\makeatother

\begin{document}



\title[A Regular Unimodular Triangulation\ldots]{A regular unimodular triangulation of reflexive 2-supported weighted projective space simplices}

\author{Benjamin Braun}
\address{Department of Mathematics\\
  University of Kentucky\\
  Lexington, KY 40506--0027}
\email{benjamin.braun@uky.edu}

\author{Derek Hanely}
\address{Department of Mathematics\\
  University of Kentucky\\
  Lexington, KY 40506--0027}
\email{derek.hanely@uky.edu}



\thanks{BB was partially supported by NSF award DMS-1953785. 
DH was partially supported by NSF award DUE-1356253.}

\date{22 October 2020}

\begin{abstract}
For each integer partition $\bq$ with $d$ parts, we denote by $\Pq$ the lattice simplex obtained as the convex hull in $\R^d$ of the standard basis vectors along with the vector $-\bq$.
For $\bq$ with two distinct parts such that $\Pq$ is reflexive and has the integer decomposition property, we establish a characterization of the lattice points contained in $\Pq$.
We then construct a Gr\"{o}bner basis with a squarefree initial ideal of the toric ideal defined by these simplices.
This establishes the existence of a regular unimodular triangulation for reflexive 2-supported $\Pq$ having the integer decomposition property. 
\end{abstract}

\maketitle



\section{Introduction \& Background}

Consider an integer partition $\bq\in \Z_{\geq 1}^d$ where $q_1\leq\cdots\leq q_d$.

\begin{definition}
  The lattice simplex associated with $\bq$ is
  \[
    \Pq := \conv{\be_1,\ldots,\be_d,-\sum_{i=1}^d q_i\be_i} \subset \R^d,
  \]
  where $\be_i$ denotes the $i$-th standard basis vector in $\R^d$.
Set $N(\bq):=1+\sum_iq_i$.
\end{definition}

It is straightforward to prove \cite[Proposition 4.4]{NillSimplices} that $N(\bq)$ is the normalized volume of $\Pq$.
  Let $\sQ$ denote the set of all lattice simplices of the form $\Pq$.
The simplices in $\sQ$ correspond to a subset of the simplices defining weighted projective spaces \cite{conrads}; the vector $(1,\bq)$ gives the weights of the associated weighted projective space.
Simplices in $\sQ$ have been the subject of active recent study~\cite{LaplacianSimplicesDigraphs,BraunDavisReflexive,BraunDavisSolusIDP,BraunLiu,LiuSolusUnimodalityPositivity,SolusNumeralSystems}.
Given a vector of distinct positive integers $\br = (r_1,\ldots,r_t)$, we write
\[
  (r_1^{x_1},r_2^{x_2},\ldots,r_t^{x_t}):=(\underbrace{r_1,r_1,\ldots,r_1}_{x_1\text{ times}},\underbrace{r_2,r_2,\ldots,r_2}_{x_2\text{ times}},\ldots,\underbrace{r_t,r_t,\ldots,r_t}_{x_t\text{ times}}) \, .
\]
There is a natural stratification of $\sQ$ based on the distinct entries in the vector $\bq$, leading to the following definition.

\begin{definition}
  If $\bq=(q_1,\ldots,q_d)=(r_1^{x_1},r_2^{x_2},\ldots,r_t^{x_t})$, we say that both $\bq$ and $\Pq$ are \emph{supported} by the vector $\br = (r_1,\ldots,r_t)$ with \emph{multiplicity} $\bx=(x_1, \dots, x_t)$.
  We write $\bq=(\br,\bx)$ in this case, and say that $\bq$ is \emph{$t$-supported}.
\end{definition}

We say $\bq$ and $\Pq$ are \emph{reflexive} if 
\begin{align*}
q_i \text{ divides } 1 + \sum_{j =1}^d q_j \quad\text{for all $1 \le i \le d$\,.}
\label{equ:conrads}
\end{align*}
When $\bq$ is reflexive, the geometric dual $\Pq^*$ of $\Pq$ is a lattice polytope, where
\[
	\Pq^* = \{ \by \in \R^d \mid \langle \ba,\by\rangle \geq -1 \text{ for all } \ba \in \Pq\} \, .
\]
Lattice polytopes whose duals are also lattice polytopes are known as reflexive polytopes, hence our nomenclature.
We say a lattice polytope $P$ has the \emph{integer decomposition property}, or \emph{is IDP}, if for every $M\in \Z_{\geq 1}$ and $p\in MP\cap \Z^d$, there exist $p_1,\ldots,p_M\in P\cap Z^d$ such that $p=p_1+\cdots+p_M$.
A detailed study of reflexive IDP $\Pq$ was initiated in~\cite{BraunDavisSolusIDP}, motivated by several open problems regarding unimodality in Ehrhart theory.
Braun, Davis, and Solus~\cite[Theorem 4.1]{BraunDavisSolusIDP} classified the 2-supported reflexive IDP $\Pq$, proving that every such $\bq$ is of the form $(r_1^{x_1},r_2^{x_2})$ where either
\begin{align*}
    r_1 &> 1 \text{ with } r_2 = 1+r_1x_1 \text{ and } x_2 = r_1 - 1, \text{ or} \\
    r_1 &= 1 \text{ with } r_2 = 1 + x_1 \text{ and } x_2 \text{ arbitrary}.
\end{align*}
They further established that every $\Pq$ has a unimodal Ehrhart $h^*$-vector.

It is well known that if a lattice polytope $P$ admits a unimodular triangulation, then $P$ is IDP.
Further, Bruns and R\"{o}mer~\cite{BrunsRomer} proved that if $P$ is reflexive and admits a regular unimodular triangulation, then $P$ has a unimodal Ehrhart $h^*$-vector.
Thus, it is of interest to determine whether or not reflexive IDP lattice polytopes admit regular unimodular triangulations.
It has been shown~\cite{BraunDavisSolusIDP} that each 2-supported reflexive IDP $\Pq$ with $\bq=(1^{x_1},(1+x_1)^{x_2})$ arises as an affine free sum of $\Delta_{(1,1^{x_1})}$ and $\Delta_{(1,1^{x_2})}$.
Thus, every $\Pq$ of this form admits a regular unimodular triangulation, for example the triangulation arising as the join of the boundary of $\Delta_{(1,1^{x_1})}\times (0^{x_2})$ with the unique unimodular triangulation of $(0^{x_1})\times \left(\Delta_{(1,1^{x_2})}-\be_{x_1+1}\right)$ in $\R^{x_1+x_2}$.

In this paper, we study regular unimodular triangulations for the other 2-supported case.
Thus, for the remainder of this paper, we assume that $\bq=(r_1^{x_1},(1+r_1x_1)^{r_1-1})$ with $r_1>1$.
Observe that $\dim \Pq = d = x_1 + x_2 = r_1 + x_1 - 1$. 
Define $\sA' = \{\ba'_1,\ldots,\ba'_{r_1+3} , \bb'_1, \ldots, \bb'_d\}\subset \Z^d$, where $\ba'_{r_1+1} = ((-1)^{x_1}, (-x_1)^{r_1-1})$, $\ba'_{r_1+2} = (0^{x_1}, (-1)^{r_1-1})$, $\ba'_{r_1+3} = (0^{x_1}, 0^{r_1-1})$, $\ba'_{i} = (r_1-i+1)\ba'_{r_1+1} + \ba'_{r_1+2}$ for $1\leq i \leq r_1$, and $\bb'_j = \be_{d-j+1}$ for $1\leq j \leq d$. (Note that $\ba'_1 = -\bq$, so all vertices of $\Pq$ are contained in $\sA'$.)

\begin{example}
Let $r_1=6$ and $x_1=4$, so $\bq = (6^4,25^5)\in \Z^9$. The elements of $\sA'$ are given by the columns of the following matrix:
\begin{align*}
    \begin{blockarray}{*{18}c}
    \ba'_1 & \ba'_2 & \ba'_3 & \ba'_4 & \ba'_5 & \ba'_6 & \ba'_7 & \ba'_8 & \ba'_9 & \bb'_1 & \bb'_2 & \bb'_3 & \bb'_4 & \bb'_5 & \bb'_6 & \bb'_7 & \bb'_8 & \bb'_9  \\
        \begin{block}{(*{18}c)}
        -6  & -5  & -4  & -3  & -2 & -1 & -1 &  \phantom{-}0 & 0 & 0 & 0 & 0 & 0 & 0 & 0 & 0 & 0 & 1 \\
        -6  & -5  & -4  & -3  & -2 & -1 & -1 &  \phantom{-}0 & 0 & 0 & 0 & 0 & 0 & 0 & 0 & 0 & 1 & 0 \\
        -6  & -5  & -4  & -3  & -2 & -1 & -1 &  \phantom{-}0 & 0 & 0 & 0 & 0 & 0 & 0 & 0 & 1 & 0 & 0 \\
        -6  & -5  & -4  & -3  & -2 & -1 & -1 &  \phantom{-}0 & 0 & 0 & 0 & 0 & 0 & 0 & 1 & 0 & 0 & 0 \\
        -25 & -21 & -17 & -13 & -9 & -5 & -4 & -1 & 0 & 0 & 0 & 0 & 0 & 1 & 0 & 0 & 0 & 0 \\
        -25 & -21 & -17 & -13 & -9 & -5 & -4 & -1 & 0 & 0 & 0 & 0 & 1 & 0 & 0 & 0 & 0 & 0 \\
        -25 & -21 & -17 & -13 & -9 & -5 & -4 & -1 & 0 & 0 & 0 & 1 & 0 & 0 & 0 & 0 & 0 & 0 \\
        -25 & -21 & -17 & -13 & -9 & -5 & -4 & -1 & 0 & 0 & 1 & 0 & 0 & 0 & 0 & 0 & 0 & 0 \\
        -25 & -21 & -17 & -13 & -9 & -5 & -4 & -1 & 0 & 1 & 0 & 0 & 0 & 0 & 0 & 0 & 0 & 0 \\
        \end{block}
    \end{blockarray}
\end{align*}
\end{example}

In this paper, we prove the following theorems.

\begin{theorem}
The lattice points of the IDP simplex $\Pq$ are given by $\sA'$. 
\label{thm:latticepoints}
\end{theorem}

\begin{theorem}
There exists a lexicographic squarefree initial ideal of the toric ideal associated with $\sA'$.
\label{thm:main}
\end{theorem}

\begin{corollary}
The convex polytope $\Pq = \conv{\sA'}$ admits a regular unimodular triangulation induced by the lexicographic term order $<_{lex}$.
\label{cor:main}
\end{corollary}

\section{Proof of Theorem~\ref{thm:latticepoints}}\label{sec:latticepoints}

First, note that by \cite[Theorem 2.2]{BraunDavisSolusIDP}, we know the \emph{Ehrhart $h^{*}$-polynomial} of $\Pq$, denoted $h^{*}(\Pq; z):= h_0^{*} + h_1^{*}z + \cdots + h_d^{*}z^d$, is given by
\begin{align*}
    h^{*}(\Pq; z) = \sum_{b=0}^{r_1(x_1r_1+1)-1} z^{w(\bq,b)},
\end{align*}
where
\begin{align*}
    w(\bq,b):= b - x_1\left\lfloor \frac{b}{1+x_1r_1} \right\rfloor - (r_1-1) \left\lfloor \frac{b}{r_1} \right\rfloor.
\end{align*}
It is well known (see, e.g., \cite{beck2007computing}) that the coefficient $h_1^{*}$ is given by the formula
\[
h_1^{*} = \abs{\Pq \cap \Z^d} - (\dim \Pq + 1) \, . 
\]
Thus, to compute $h_1^{*}$, we must determine all $b$ for which $w(\bq,b) = 1$.
Since $0\leq b \leq r_1(x_1r_1+1)-1$, the division algorithm allows us to write $b = \alpha(1+x_1r_1)+\beta$, where $0\leq \alpha < r_1$ and $0\leq \beta < 1+x_1r_1$. Hence,
\begin{align*}
    w(\bq,b) &= w(\bq, \alpha(1+x_1r_1)+\beta) \\
    &= \alpha(1+x_1r_1) + \beta - x_1\left\lfloor \frac{\alpha(1+x_1r_1)+\beta}{1+x_1r_1} \right\rfloor - (r_1-1)\left\lfloor \frac{\alpha(1+x_1r_1)+\beta}{r_1} \right\rfloor \\
    &= \alpha(1+x_1r_1) + \beta - \alpha x_1 - (r_1-1)\left(\alpha x_1 + \left\lfloor \frac{\alpha+\beta}{r_1} \right\rfloor \right) \\
    &= \alpha + \beta - (r_1-1)\left\lfloor \frac{\alpha + \beta}{r_1}\right\rfloor.
\end{align*}
Therefore, the equation $w(\bq,b)=1$ becomes 
\begin{align*}
    \alpha + \beta - (r_1-1)\left\lfloor \frac{\alpha + \beta}{r_1} \right\rfloor = 1 \quad \Longleftrightarrow \quad  \alpha+\beta = 1 + (r_1-1)\left\lfloor \frac{\alpha + \beta}{r_1} \right\rfloor.
\end{align*}
Now, let $\ell = \left\lfloor \frac{\alpha + \beta}{r_1} \right\rfloor$.
By the previous equation, $\alpha + \beta = 1+(r_1-1)\ell$.
Substituting this equivalent expression for $\alpha + \beta$ into both sides of the previous equation, it follows that solving $w(\bq,b)=1$ is equivalent to finding all pairs $(\alpha,\beta)$ such that
\begin{align*}
    1 + (r_1-1)\ell &= 1 + (r_1-1)\left\lfloor \frac{1+(r_1-1)\ell}{r_1} \right\rfloor \\
    &= 1 + (r_1-1)\left(\ell + \left\lfloor \frac{1-\ell}{r_1} \right\rfloor \right).
\end{align*}
Rearranging this equation yields
\begin{align*}
    (r_1-1)\left\lfloor \frac{1-\ell}{r_1} \right\rfloor = 0.
\end{align*}
Therefore, since $r_1 >1$, this implies 
\begin{align*}
    \left\lfloor \frac{1-\ell}{r_1} \right\rfloor = 0 \quad \Longrightarrow \quad \ell = \begin{cases} 0 \\ 1 \end{cases} 
    \Longrightarrow \quad \alpha+\beta = \begin{cases} 1 \\ r_1 \end{cases}.
\end{align*}
If $\alpha + \beta = 1$, then $(\alpha,\beta) = (1,0)$ or $(\alpha,\beta) = (0,1)$. Otherwise, in the case that $\alpha+\beta = r_1$, there are $r_1$ possible pairs $(\alpha,\beta)$ where $\alpha \in \{0,\ldots,r_1-1\}$ and $\beta = r_1 - \alpha$. Thus,
\begin{align*}
    h_1^{*} = \abs{\{b \, : \, w(\bq,b) = 1\}} = r_1+2.
\end{align*}
Consequently, the aforementioned formula for $h_1^{*}$ implies
\begin{align*}
    \abs{\Pq \cap \Z^d} = r_1 + d +3,
\end{align*}
that is, there are $r_1+d+3$ lattice points in $\Pq$.

\begin{prop}
For $\bt = (t_1,\ldots,t_d)\in \R^d$, define
\begin{align*}
\lambda_k(\bt) := 
\begin{cases}
    \ds\;  \sum_{\substack{j=1 \\ j\neq k}}^d t_j - x_1r_1t_k, &\text{ if } 1\leq k \leq x_1 \\
    \ds\; \sum_{\substack{j=1 \\ j\neq k}}^d t_j - (r_1 - 1)t_k, &\text{ if } x_1+1\leq k \leq d \\
    \ds\; \sum_{j=1}^{d} t_j, &\text{ if } k=d+1 \, .
\end{cases}
\end{align*} 
An irredundant $\mathcal{H}$-description of $\Pq$ is given by $\lambda_k(\bt) \leq 1$ for all $1\leq k\leq d+1$. 
\label{prop:ineqs}
\end{prop}
\begin{proof}
Observe that for all $1\leq j \leq d$, $\be_j$ satisfies all of the given inequalities tightly except when $k = j$ (i.e., $\lambda_k(\be_j) = 1$ for all $k\neq j$ and $\lambda_j(\be_j) < 1$). 
Moreover, $-\bq$ satisfies the first $d$ inequalities tightly (i.e., $\lambda_k(-\bq) = 1$ for all $1\leq k \leq d$), but not $\sum_j t_j \leq 1$. 
Thus, as each vertex of the simplex $\Pq$ satisfies exactly $d$ of the given inequalities with equality, the inequalities necessarily constitute an $\mathcal{H}$-description of $\Pq$.
\end{proof}

\begin{proof}[Proof of Theorem \ref{thm:latticepoints}]
To begin, observe that $\abs{\sA'} = r_1+d+3 = \abs{\Pq \cap \Z^d}$. Therefore, as each element in $\sA'$ is an integer vector, it suffices to show that each point satisfies the inequalities in Proposition \ref{prop:ineqs}. To this end, let $\lambda_k$ be defined as in Proposition \ref{prop:ineqs};
we evaluate each vector in $\sA'$ on $\lambda_k$. 
For each $1\leq i \leq r_1$, note that
\begin{align*}
    \ba'_i = (r_1-i+1)\ba'_{r_1+1} + \ba'_{r_1+2} = \left((-(r_1-i+1))^{x_1},((-(1+(r_1-i+1)x_1))^{r_1-1}\right).
\end{align*}
Therefore, we have that
\begin{align*}
    \lambda_k(\ba'_1) = 1 \text{ if } 1\leq k \leq d \quad \text{and} \quad \lambda_{d+1}(\ba'_1) < 1,
\end{align*}
and for each $i \in \{2,\ldots, r_1\}\cup\{r_1+2\}$,
\begin{align*}
    \lambda_k(\ba'_i) = 1 \text{ if } x_1+1 \leq k \leq d \quad \text{and} \quad \lambda_k(\ba'_i) < 1 \text{ otherwise}.
\end{align*}
Also,
\begin{align*}
    \lambda_k(\ba'_{r_1+1}) = 1 \text{ if } 1 \leq k \leq x_1 \quad \text{and} \quad \lambda_k(\ba'_{r_1+1}) < 1 \text{ otherwise},
\end{align*}
and
\begin{align*}
    \lambda_k(\ba'_{r_1+3}) < 1 \text{ for all } 1\leq k \leq d+1.
\end{align*}
Lastly, for all $1\leq j \leq d$,
\begin{align*}
    \lambda_k(\bb'_j) = 1 \text{ if } k\neq d-j+1 \quad \text{and} \quad \lambda_k(\bb'_j) < 1 \text{ if } k = d-j+1.
\end{align*}
Thus, $\sA' \subseteq \Pq \cap \Z^d$, and the result follows.  
\end{proof}


\section{Proof of Theorem~\ref{thm:main}}

We next seek to prove the existence of a regular unimodular triangulation of the convex hull of these points.
Given a field $K$, there are natural parallels between properties of lattice polytopes and algebraic objects such as semigroup algebras, toric varieties, and monomial ideals.
The following one-to-one correspondence between lattice points and Laurent monomials plays a central role:
\begin{align*}
    \ba' = (a_1,\ldots,a_d)\in \Z^d \quad \longleftrightarrow \quad \bt^{\ba'}:= t_1^{a_1}\cdots t_d^{a_d}\in K[t_1^{\pm 1},\ldots,t_d^{\pm 1}].
\end{align*}
For details regarding the significance of this correspondence, see \cite[Chapter 8]{sturmfels}. Furthermore, for all notation related to combinatorical commutative algebra, we refer the reader to \cite{coxlittleoshea}.

Let $K$ be a field, and define $\sA = (\ba_1,\ldots,\ba_{r_1+3}, \bb_1, \ldots, \bb_d)\subset \Z^{(d+1)\times (r_1+d+3)}$ to be the homogenization of $\sA'$ where $\ba_i = (\ba'_i, 1)$ and $\bb_j = (\bb'_j,1)$; that is, $\sA$ is the matrix associated with the point configuration consisting of all vectors in $\sA'$ lifted to height 1. 
(Note that we can view the columns of $\sA$ as the intersection of $\Z^{d+1}$ with the degree $1$ slice of the polyhedral cone over $\Pq$.)
Let $K[\sA] := K[\{z_1,\ldots,z_{r_1+3},y_1,\ldots, y_d\}]$ be the polynomial ring associated with the columns of $\sA$ in $r_1+d+3$ variables over $K$.
Moreover, let $\sM(K[\sA])$ denote the set of monomials contained in $K[\sA]$, and let $\sR_K[\sA]$ be the $K$-subalgebra of the Laurent polynomial ring $K[\bt^{\pm 1}]:= K[t_1^{\pm 1},\ldots,t_{d+1}^{\pm 1}]$ generated by all monomials $\bt^{\ba}$ with $\ba\in \sA$, where $\bt^{\ba} = t_1^{a_1}\cdots t_{d+1}^{a_{d+1}}$ if $\ba = (a_1,\ldots,a_{d+1})$. 
The toric ideal $I_{\sA}$ is the kernel of the surjective ring homomorphism $\pi: K[\sA] \to \sR_K[\sA]$ defined by
\begin{align*}
	\pi(z_i) &= \bt^{\ba_i}, \, \text{ for } 1 \leq i \leq r_1+3 \\
	\pi(y_j) &= \bt^{\bb_j}, \text{ for } 1 \leq j \leq d.
\end{align*}
A generating set for $I_{\sA}$ is given by the set of all homogeneous binomials $f-g$ with $\pi(f)=\pi(g)$ and $f,g\in \sM(K[\sA])$, see \cite[Lemma 4.1]{sturmfels}.
We fix the term order $<_{lex}$ on $K[\sA]$ induced by the ordering of the variables
\begin{align*}
	z_1 > z_2 > \cdots > z_{r_1+3} > y_1 > y_2 > \cdots > y_d. 
\end{align*}
Moreover, for $f = z_1^{\gamma_1}\cdots z_{r_1+3}^{\gamma_{r_1+3}}y_1^{\delta_1}\cdots y_d^{\delta_d} \in \sM(K[\sA])$, we introduce the notation 
\[
\zsupp{f}:= \{i\in \{1,\ldots ,r_1+3\} \, : \, \gamma_i > 0 \} \, .
\]

Given this setup, we restate Corollary~\ref{cor:main} and Theorem~\ref{thm:main}. 
The proof of Corollary~\ref{cor:restated} follows immediately from Theorem~\ref{thm:basis} due to the existence of a squarefree initial ideal of the toric ideal $I_{\sA}$~\cite[Corollary 8.9]{sturmfels}.

\begin{corollary}[Restatement of Corollary \ref{cor:main}]
The regular triangulation of $\sA$ induced by the term order $<_{lex}$ as specified above is unimodular. 
\label{cor:restated}
\end{corollary}

\begin{theorem}[Restatement of Theorem \ref{thm:main}]
Let $B$ be the set of all $(i,j)\in \N^2$ satisfying the following conditions:
\begin{enumerate}[label={(\roman*)}]
	\item $j-i \geq 2$
	\item $1\leq i \leq r_1$
	\item $j \leq r_1+3$
	\item $j \neq r_1+1$
\end{enumerate}
Given $(i,j)\in B$, define $(k,\ell)$ as follows:
\begin{align*}
\begin{array}{l l}
\ds k = \left\lfloor \frac{i+j}{2} \right\rfloor, \ell = \left\lceil \frac{i+j}{2} \right\rceil & \text{ if } j < r_1+1 \\[2ex]
\ds k = \left\lfloor \frac{i+j-1}{2} \right\rfloor, \ell = \left\lceil \frac{i+j-1}{2} \right\rceil & \text{ if } j = r_1+2\\[2ex]
k = i+1, \, \ell = r_1+1 & \text{ if } j = r_1+3, i\neq r_1 \\[1ex]
k = r_1+1, \, \ell = r_1+2 & \text{ if } j = r_1+3, i=r_1.
\end{array}
\end{align*}
If $x_1 \geq r_1-2$, then the set of binomials $\sG$ given by
\begin{alignat}{2}
&z_i z_j - z_k z_{\ell}, && \quad (i,j)\in B \label{eq:gb1}\\
&\ds z_{k+1}\prod_{\ell=1}^{r_1 - 1}y_{\ell} - z_{r_1 + 1}^{r_1 - k}z_{r_1 + 3}^k, && \quad 0\leq k \leq r_1 - 1 \label{eq:gb2}\\ 
&\ds z_{r_1 - k}\prod_{\ell=r_1}^{d}y_{\ell} - z_{r_1}^{k}z_{r_1 + 2}^{x_1+1 - k}, && \quad 0\leq k \leq r_1 - 1 \label{eq:gb3}\\
&\ds z_{r_1+2}\prod_{\ell=1}^{r_1 - 1}y_{\ell} - z_{r_1 + 3}^{r_1}, &&  \label{eq:gb4}\\ 
&\ds z_{r_1 + 1}\prod_{\ell=r_1}^{d}y_{\ell} - z_{r_1 + 2}^{x_1}z_{r_1 + 3} && \label{eq:gb5}
\end{alignat}
is a Gr\"{o}bner basis of $I_{\sA}$ with respect to lexicographic order $<_{lex}$. 
In the case that $x_1<r_1-2$, replace (3) above with
\begin{align}
    \begin{cases}
        \ds z_{r_1 - k}\prod_{\ell=r_1}^{d}y_{\ell} - z_{r_1}^{k}z_{r_1 + 2}^{x_1+1 - k}, \quad & 0\leq k \leq x_1+1 \\
        \ds z_{r_1 - k}\prod_{\ell=r_1}^{d}y_{\ell} - z_{r_1-1}^{k-x_1-1}z_{r_1}^{2x_1+2 - k}, \quad & x_1+2\leq k \leq r_1-1.
    \end{cases}
    \tag{3*}
\end{align}
\label{thm:basis}
\end{theorem}

Note that regardless of case (either $x_1\geq r_1 - 2$ or $x_1 < r_1-2$), the initial terms of the $k$-th binomial in (3) and (3*) are identical. 
Therefore, we need not worry about the relationship between $x_1$ and $r_1-2$.

To prove Theorem \ref{thm:basis}, we employ the following well known technique for proving a finite set of the toric ideal $I_{\sA}$ is a Gr\"{o}bner basis of $I_{\sA}$ (see, e.g., \cite[(0.1)]{hibiohsugi}).

\begin{lemma}
A finite set $\sG$ of $I_{\sA}$ is a Gr\"{o}bner basis of $I_{\sA}$ with respect to the monomial order $<$ if and only if $\{\pi(f) \, : \, f \in \sM(K[\sA]), f\notin in_{<}(\sG)\}$ is linearly independent over $K$; 
i.e., if and only if $\pi(f) \neq \pi(g)$ for all $f\notin in_{<}(\sG)$ and $g\notin in_{<}(\sG)$ with $f \neq g$.
\label{lem:grobner}
\end{lemma}

We will also require the following fact which provides an upper bound on the supported $z$ variables for any monomial outside the initial ideal generated by the binomials in Theorem \ref{thm:basis} with respect to $<_{lex}$. 

\begin{lemma}
Let $\sG$ be the set of binomials given in Theorem \ref{thm:basis}.
Suppose 
\begin{align*}
    f = z_1^{\gamma_1}\cdots z_{r_1+3}^{\gamma_{r_1+3}}y_1^{\delta_1}\cdots y_d^{\delta_d} \in \sM(K[\sA])
\end{align*} 
with $f\notin in_{<_{lex}}(\sG)$ and $\abs{\zsupp{f}} \geq 1$. Let $m$ denote the minimal index such that $z_m$ divides $f$. Then, $\abs{\zsupp{f}} \leq 3$ and we are restricted to the following possibilities:
\begin{enumerate}[label={(\arabic*)}]
	\item if $1\leq m \leq r_1-1$, then $\gamma_{m+1},\gamma_{r_1+1}\geq 0$ and \\ $\gamma_i = 0$ for all $i\in \{1,\ldots,r_1+3\} \setminus \{m, m+1, r_1+1\}$.
	\item if $m=r_1$, then $\gamma_{r_1+1},\gamma_{r_1+2}\geq 0$ and \\ $\gamma_i = 0$ for all $i\in \{1,\ldots,r_1-1\} \cup \{r_1+3\}$.
	\item if $m\in \{r_1+1,r_1+2,r_1+3\}$, then $\gamma_{i} = 0$ for all $i < m$ and \\ $\gamma_i \geq 0$ for all $i > m$. 
\end{enumerate}
\label{lem:3support}
\end{lemma}
\begin{proof}
Suppose $1\leq m \leq r_1-1$. Since $z_mz_{m+1}\notin \inG$ and $z_mz_{r_1+1}\notin \inG$, $z_{m+1}$ and $z_{r_1+1}$ possibly divide $f$. However, given the structure of $B$ as defined in Theorem \ref{thm:basis}, it follows that $z_mz_{r_1+2},z_mz_{r_1+3},z_mz_n \in \inG$ for all $n$ with $n > m+1, n \neq r_1+1$. Therefore, since $m$ is minimal, $\abs{\zsupp{f}} \leq 3$ and we precisely satisfy the conditions of Lemma~\ref{lem:3support}(1). \\
Now, suppose $m = r_1$. By the minimality of $m$, we need only consider indices greater than $r_1$. Observe that $z_{r_1}z_{r_1+1} \notin \inG$, $z_{r_1}z_{r_1+2} \notin \inG$, and $z_{r_1}z_{r_1+3} \in \inG$. Thus, we have that $\abs{\zsupp{f}} \leq 3$ and we end up in Lemma~\ref{lem:3support}(2). \\
Finally, for $m\in \{r_1+1,r_1+2,r_1+3\}$, minimality of $m$ immediately implies $\abs{\zsupp{f}} \leq 3$. To see that this case yields Lemma~\ref{lem:3support}(3), observe that $z_{m}z_{n} \notin \inG$ for $m,n\in \{r_1+1,r_1+2,r_1+3\}$ with $m\neq n$. 
\end{proof}

\begin{proof}[Proof of Theorem \ref{thm:basis}]
One easily checks that each binomial $h = m_1 - m_2 \in \sG$ is contained in $I_{\sA}$ by showing $\pi(m_1) = \pi(m_2)$. To show $\sG$ is a Gr\"{o}bner basis of $I_{\sA}$, we employ Lemma \ref{lem:grobner}. Suppose $f,g\in \sM(K[\sA])$ with $f\neq g$, $f\notin \inG$, and $g\notin \inG$. Write 
\begin{align*}
    f = z_1^{\alpha_1}\cdots z_{r_1+3}^{\alpha_{r_1+3}}y_1^{\beta_1}\cdots y_d^{\beta_d} \quad \text{and} \quad g = z_1^{\alpha'_1}\cdots z_{r_1+3}^{\alpha'_{r_1+3}}y_1^{\beta'_1}\cdots y_d^{\beta'_d},
\end{align*}
where $\alpha_i, \alpha'_i, \beta_j, \beta'_j \geq 0$. 
We may assume $f$ and $g$ are relatively prime (since otherwise, we could simply factor out the common variables and consider the images of the reduced monomials). 
Further assume to the contrary that $\pi(f) = \pi(g)$, and without loss of generality, assume $\abs{\zsupp{f}} \geq \abs{\zsupp{g}}$. 
For convenience, let $\bbf,\bbg \in \Z^{d+1}$ denote the exponent vectors associated with $\pi(f)$ and $\pi(g)$, respectively, and let $\bbf[k]$ (resp. $\bbg[k]$) denote the $k$-th entry of $\bbf$ (resp. $\bbg$). 
With this notation, observe that $\pi(f) = \pi(g)$ if and only if $\bbf[k] = \bbg[k]$ for all $1\leq k \leq d+1$.
Now, we verify a contradiction in each of the following cases.\\
\begin{case}
$\abs{\zsupp{g}} = 0$. By definition, it follows that $\alpha'_i = 0$ for all $1\leq i \leq r_1+3$. 
Therefore, we know that
\begin{align*}
    \bbg = \begin{pmatrix}\beta'_d, \ldots, \beta'_1, \sum_j \beta'_j \end{pmatrix}.
\end{align*}
    \begin{subcase} $\abs{\zsupp{f}} = 0$.
        Thus, 
        \begin{align*}
            \bbf = \begin{pmatrix}\beta_d, \ldots, \beta_1, \sum_j \beta_j \end{pmatrix}.
        \end{align*}
        Since $\pi(f)=\pi(g)$, this implies $\beta_j = \beta'_j$ for all $1\leq j \leq d$, and consequently, $f=g$, a contradiction.
    \end{subcase}
    \begin{subcase} $\abs{\zsupp{f}} \geq 1$.
        Let $m$ denote the minimal index such that $z_m$ divides $f$ (i.e., $\alpha_m > 0$ and $\alpha_i = 0$ for all $i < m$). 
        \begin{enumerate}[label={\textbf{(\alph*)}},leftmargin=.575in]
            \item Suppose $1\leq m \leq r_1+1$. 
            Since $z_my_{r_1}\cdots y_d\in \inG$ (by (\ref{eq:gb3}) and (\ref{eq:gb5})) and $f \notin \inG$, there exists an index $\ell\in \{r_1,\ldots,d\}$ such that $\beta_{\ell} = 0$. Hence, 
            \begin{align*}
                \bbf[d-\ell+1] = \underbrace{\sum_{i=1}^{r_1+3}\alpha_i \sA_{d-\ell+1,i}}_{<\,0} + \underbrace{\sum_{j=1}^{d} \beta_j \sA_{d-\ell+1,r_1+3+j}}_{=\,0} < 0.
            \end{align*}
            However, $\bbg[d-\ell+1]=\beta'_{\ell} \geq 0$, a contradiction.
            \item Suppose $m = r_1+2$. Since $z_{r_1+2}y_1\cdots y_{r_1-1} \in \inG$ (by (\ref{eq:gb4})) and $f\notin \inG$, there exists an index $k\in\{1,\ldots,r_1-1\}$ such that $\beta_k=0$. 
            Since $k<r_1$, it follows that $d-k+1 > x_1$. 
            Therefore, $\sA_{d-k+1,r_1+2} = -1$.
            Hence,
            \begin{align*}
                \bbf[d-k+1] = \underbrace{\sum_{i=1}^{r_1+3}\alpha_i \sA_{d-k+1,i}}_{= \, -\alpha_{r_1+2}\,<\,0} + \underbrace{\sum_{j=1}^{d} \beta_j \sA_{d-k+1,r_1+3+j}}_{=\,0} < 0.
            \end{align*}
            However, $\bbg[d-k+1]=\beta'_{k} \geq 0$, a contradiction.
            \item Suppose $m = r_1+3$. 
            Since $m$ is minimal, we know $\alpha_i = 0$ for all $1\leq i \leq r_1+2$. Since $\sA$ is homogenized, we also know $\sum_i \alpha_i + \sum_j \beta_j = \sum_i \alpha'_i + \sum_j \beta'_j$ (this can be seen directly from $\bbf[d+1] = \bbg[d+1]$).
            Hence, in this case, the equation simplifies to $\alpha_{r_1+3} + \sum_{j} \beta_j = \sum_{j} \beta'_j$, and moreover,
            \begin{align*}
                \textstyle \bbf = (\beta_d,\ldots,\beta_1,\alpha_{r_1+3} + \sum_{j} \beta_j).
            \end{align*}
            Since $\pi(f)=\pi(g)$, $\beta_j = \beta'_j$ for all $1\leq j \leq d$. Therefore, substituting into the above equation,
            \begin{align*}
                \textstyle \alpha_{r_1+3} + \sum_{j} \beta_j = \sum_{j} \beta'_j = \sum_{j} \beta_j,
            \end{align*}
            but $\alpha_{r_1+3}>0$, a contradiction. 
        \end{enumerate}
    \end{subcase}
\end{case}
\begin{case}
$\abs{\zsupp{g}} \geq 1$. Let $n$ denote the minimal index such that $z_n$ divides $g$ (i.e., $\alpha'_n > 0$ and $\alpha'_i = 0$ for all $i<n$). 
Since $\abs{\zsupp{f}} \geq \abs{\zsupp{g}}$ and $\abs{\zsupp{g}}\geq 1$, we know $\zsupp{f} \neq \emptyset$. 
Hence, let $m$ denote the minimal index such that $z_m$ divides $f$. 
Via Lemma \ref{lem:3support}, this case naturally lends itself to the following subcases of consideration.
    \begin{subcase} $n\in \{1,\ldots,r_1-1\}$.
        By Lemma $\ref{lem:3support}$, we know $\alpha'_n > 0$, $\alpha'_{n+1},\alpha'_{r_1+1} \geq 0$, and $\alpha'_i = 0$ for all $i\in \{1,\ldots,r_1+3\} \setminus \{n,n+1,r_1+1\}$. 
        Since $z_ny_1\cdots y_{r_1-1}\in \inG$ (by (\ref{eq:gb2})), $z_ny_{r_1}\cdots y_{d}\in \inG$ (by (\ref{eq:gb3})), and $g\notin \inG$, there exist indices $k_1\in \{1,\ldots, r_1-1\}$ and $\ell_1\in \{r_1,\ldots, d\}$ such that $\beta'_{k_1} = \beta'_{\ell_1} = 0$. 
        Then,
        \begin{align}
            \bbf[d-k_1+1] &= \sum_{i=1}^{r_1+3}\alpha_i\sA_{d-k_1+1,i} + \beta_{k_1} \label{eq:case2.1:fk1} \\
            \bbg[d-k_1+1] &= -(1+(r_1-n+1)x_1)\alpha'_n - (1+(r_1-n)x_1)\alpha'_{n+1} - x_1\alpha'_{r_1+1} \label{eq:case2.1:gk1} \\
            \bbf[d-\ell_1+1] &= \sum_{i=1}^{r_1+3}\alpha_i\sA_{d-\ell_1+1,i} + \beta_{\ell_1} \label{eq:case2.1:fl1} \\
            \bbg[d-\ell_1+1] &= -(r_1-n+1)\alpha'_n - (r_1-n)\alpha'_{n+1} - \alpha'_{r_1+1}. \label{eq:case2.1:gl1}
        \end{align}
        Note that $\pi(f)=\pi(g)$ implies $(\ref{eq:case2.1:fk1}) = (\ref{eq:case2.1:gk1})$ and $(\ref{eq:case2.1:fl1}) = (\ref{eq:case2.1:gl1})$.
        Now, we claim $m\in \{1,\ldots,r_1+1\}$. 
        Indeed, assume otherwise, that is, $\zsupp{f} \subseteq \{r_1+2,r_1+3\}$. 
        Then, $\bbf[d-\ell+1] = \beta_{\ell} \geq 0$ for all $\ell\in \{r_1,\ldots, d\}$, but from (\ref{eq:case2.1:gl1}), $\bbg[d-\ell_1+1] < 0$ since $\alpha'_n >0$ and $\alpha'_{n+1},\alpha'_{r_1+1}\geq 0$. 
        This contradicts $\pi(f) = \pi(g)$. 
        Hence, given the structure of Lemma \ref{lem:3support}, we consider the following subsubcases.
        \begin{enumerate}[label={\textbf{(\alph*)}},leftmargin=.575in]
            \item $m\in \{1,\ldots,r_1-1\}$. 
            Since $z_my_1\cdots y_{r_1-1}$ (by (\ref{eq:gb2})), $z_my_{r_1}\cdots y_d \in \inG$ (by (\ref{eq:gb3})), and $f\notin \inG$, there exist indices $k_2\in \{1,\ldots,r_1-1\}$ and $\ell_2\in \{r_1,\ldots,d\}$ such that $\beta_{k_2} = \beta_{\ell_2} = 0$.
            Then, we have that
            \begin{align}
                \bbf[d-k_2+1] &= \sum_{i=1}^{r_1+3}\alpha_i\sA_{d-k_2+1,i} \label{eq:case2.1a:fk2} \\
                \begin{split} \bbg[d-k_2+1] &= -(1+(r_1-n+1)x_1)\alpha'_n - (1+(r_1-n)x_1)\alpha'_{n+1} \\ &\qquad - x_1\alpha'_{r_1+1} + \beta'_{k_2} \end{split} \label{eq:case2.1a:gk2} \\
                \bbf[d-\ell_2+1] &= \sum_{i=1}^{r_1+3}\alpha_i\sA_{d-\ell_2+1,i} \label{eq:case2.1a:fl2} \\
                \bbg[d-\ell_2+1] &= -(r_1-n+1)\alpha'_n - (r_1-n)\alpha'_{n+1} - \alpha'_{r_1+1} + \beta'_{\ell_2}, \label{eq:case2.1a:gl2}
            \end{align}
            where $(\ref{eq:case2.1a:fk2}) = (\ref{eq:case2.1a:gk2})$ and $(\ref{eq:case2.1a:fl2}) = (\ref{eq:case2.1a:gl2})$ as $\pi(f)=\pi(g)$. 
            Since $1\leq k_i\leq r_1-1$ for $i\in\{1,2\}$, subtracting the equation $(\ref{eq:case2.1a:fk2}) = (\ref{eq:case2.1a:gk2})$ from $(\ref{eq:case2.1:fk1}) = (\ref{eq:case2.1:gk1})$ implies $\beta_{k_1} = -\beta'_{k_2}$.
            Similarly, since $r_1\leq \ell_i\leq d$ for $i\in\{1,2\}$, subtracting equation $(\ref{eq:case2.1a:fl2}) = (\ref{eq:case2.1a:gl2})$ from $(\ref{eq:case2.1:fl1}) = (\ref{eq:case2.1:gl1})$ implies $\beta_{\ell_1}=-\beta'_{\ell_2}$.
            Since $\beta_j,\beta'_j\geq 0$ for all $1\leq j \leq d$, this implies $\beta_{k_1} = \beta'_{k_2} = \beta_{\ell_1} = \beta'_{\ell_2} = 0$. 
            Also, by Lemma \ref{lem:3support}, we know $\alpha_m>0$, $\alpha_{m+1},\alpha_{r_1+1}\geq 0$, and $\alpha_i= 0$ for all $i\in \{1,\ldots,r_1+3\}\setminus \{m,m+1,r_1+1\}$.
            Consequently, equations (\ref{eq:case2.1:fk1}) and (\ref{eq:case2.1:fl1}) simplify to
            \begin{align}
                \bbf[d-k_1+1] &= -(1+(r_1-m+1)x_1)\alpha_m - (1+(r_1-m)x_1)\alpha_{m+1} - x_1\alpha_{r_1+1} \label{eq:case2.1a:fk1} \\
                \bbf[d-\ell_1+1] &=  -(r_1-m+1)\alpha_m - (r_1-m)\alpha_{m+1} - \alpha_{r_1+1}. \label{eq:case2.1a:fl1} 
            \end{align}
            Since $\pi(f)=\pi(g)$, $(\ref{eq:case2.1a:fk1}) = (\ref{eq:case2.1:gk1})$ and $(\ref{eq:case2.1a:fl1}) = (\ref{eq:case2.1:gl1})$, thereby implying $x_1(\ref{eq:case2.1:gl1}) - (\ref{eq:case2.1:gk1}) = x_1(\ref{eq:case2.1a:fl1}) - (\ref{eq:case2.1a:fk1})$.
            Observe that $x_1(\ref{eq:case2.1:gl1}) - (\ref{eq:case2.1:gk1}) = x_1(\ref{eq:case2.1a:fl1}) - (\ref{eq:case2.1a:fk1})$ is the following
            \begin{align}
                \alpha_m+\alpha_{m+1} = \alpha'_n + \alpha'_{n+1}. \label{eq:case2.1a:star}
            \end{align}
            Now, consider the equation $(\ref{eq:case2.1a:fl1}) = (\ref{eq:case2.1:gl1})$:
            \begin{align*}
                -(r_1-m+1)\alpha_m - (r_1-m)\alpha_{m+1} - \alpha_{r_1+1} = -(r_1-n+1)\alpha'_n - (r_1-n)\alpha'_{n+1} - \alpha'_{r_1+1}.
            \end{align*}
            Adding (\ref{eq:case2.1a:star}) to this equation $r_1$ times yields
            \begin{align*}
                (m-1)\alpha_m + m\alpha_{m+1}-\alpha_{r_1+1} = (n-1)\alpha'_n + n\alpha'_{n+1}-\alpha'_{r_1+1}.
            \end{align*}
            Either $m < n$ or $m > n$ (note that $m\neq n$ since $f$ and $g$ are relatively prime).
            First, suppose $m<n$. Subtracting (\ref{eq:case2.1a:star}) from our previous equation $m-1$ times gives
            \begin{align}
                \alpha_{m+1} - \alpha_{r_1+1} = (n-m)\alpha'_{n} + (n-m+1)\alpha'_{n+1} - \alpha'_{r_1+1} \label{eq:case2.1a:2star}
            \end{align}
            As $m<n$, we have that
            \begin{align*}
                \alpha_{m+1} - \alpha_{r_1+1} &= \underbrace{(n-m)}_{>\, 0}\underbrace{\alpha'_{n}}_{>\, 0} + \underbrace{(n-m+1)}_{>\, 0}\underbrace{\alpha'_{n+1}}_{\geq \, 0} - \alpha'_{r_1+1} \\
                &> \alpha'_{n}+\alpha'_{n+1} - \alpha'_{r_1+1} \\
                &\stackrel{(\ref{eq:case2.1a:star})}{=} \alpha_m + \alpha_{m+1} - \alpha'_{r_1+1},
            \end{align*}
            which implies
            \begin{align*}
                \alpha'_{r_1+1} > \underbrace{\alpha_m}_{>\, 0} + \alpha_{r_1+1} \quad \Longrightarrow \quad \alpha'_{r_1+1} >0.
            \end{align*}
            Since $f$ and $g$ are relatively prime, this forces $\alpha_{r_1+1} = 0$. Thus, $\zsupp{f}\subseteq \{m,m+1\}$.
            Moreover, $\alpha'_{n+1} = 0$ since $\abs{\zsupp{f}}\geq \abs{\zsupp{g}}$ and we have found $\alpha'_n,\alpha'_{r_1+1} > 0$. 
            Consequently, (\ref{eq:case2.1a:star}) reduces to $\alpha'_n = \alpha_m + \alpha_{m+1}$ and (\ref{eq:case2.1a:2star}) reduces to 
            \begin{align}
                \alpha'_{r_1+1} = (n-m)\alpha_m + (n-m-1)\alpha_{m+1}. \label{eq:case2.1a:3star}
            \end{align}
            Now, $\bbf[d+1] = \bbg[d+1]$ gives that
            \begin{align*}
                \alpha_m + \alpha_{m+1} + \sum_j \beta_j = \alpha'_n + \alpha'_{r_1+1} + \sum_j \beta'_j.
            \end{align*}
            Since $\alpha'_n = \alpha_m + \alpha_{m+1}$ and $\alpha'_{r_1+1}>0$, this implies $\sum_j \beta_j > \sum_j \beta'_j$. 
            For each $r_1\leq j \leq d$, $-\bbf[d-j+1] = -\bbg[d-j+1]$ is given by
            \begin{align*}
                (r_1-m+1)\alpha_m + (r_1-m)\alpha_{m+1} - \beta_j = (r_1-n+1)\alpha'_n + \alpha'_{r_1+1} -\beta'_j.
            \end{align*}
            Solving for $\alpha'_{r_1+1}$ and substituting $\alpha'_n = \alpha_m + \alpha_{m+1}$ yields
            \begin{align*}
                \alpha'_{r_1+1} = (n-m)\alpha_m + (n-m-1)\alpha_{m+1}+\beta'_j - \beta_j.
            \end{align*}
            Adding these equations for each $r_1\leq j \leq d$ gives
            \begin{align}
                \begin{split} (d-r_1+1)\alpha'_{r_1+1}  &= (d-r_1+1)\left[(n-m)\alpha_m + (n-m-1)\alpha_{m+1}\right] \\  &\qquad +\sum_{r_1\leq j\leq d}(\beta'_j - \beta_j). \end{split} \label{eq:case2.1a:dag}
            \end{align}
            Similarly, for each $1\leq j \leq r_1-1$, $-\bbf[d-j+1] = -\bbg[d-j+1]$ is given by
             \begin{align*}
                (1+(r_1-m+1)x_1)\alpha_m + (1+(r_1-m)x_1)\alpha_{m+1} - \beta_j = (1+(r_1-n+1)x_1)\alpha'_n + x_1\alpha'_{r_1+1} -\beta'_j.
            \end{align*}
            Solving for $x_1\alpha'_{r_1+1}$ and making the appropriate substitutions yields
            \begin{align*}
                x_1\alpha'_{r_1+1} = (n-m)x_1\alpha_m + (n-m-1)x_1\alpha_{m+1}+\beta'_j - \beta_j.
            \end{align*}
            Adding these equations for each $1\leq j \leq r_1-1$ gives
            \begin{align}
                \begin{split} (r_1-1)x_1\alpha'_{r_1+1} &= (r_1-1)\left[(n-m)x_1\alpha_m + (n-m-1)x_1\alpha_{m+1}\right] \\ &\qquad +\sum_{1\leq j\leq r_1-1}(\beta'_j - \beta_j). \end{split} \label{eq:case2.1a:ddag}
            \end{align}
            Combining (\ref{eq:case2.1a:dag}) and (\ref{eq:case2.1a:ddag}) gives
            \begin{align*}
                r_1x_1\alpha'_{r_1+1} = r_1x_1\underbrace{\left[(n-m)\alpha_m + (n-m-1)\alpha_{m+1}\right]}_{=\, \alpha'_{r_1+1} \text{ by } (\ref{eq:case2.1a:3star})} + \underbrace{\sum_{j=1}^{d} (\beta'_j - \beta_j)}_{<\, 0},
            \end{align*}
            a contradiction.
            Now, suppose $m > n$. 
            In this case, rather than subtracting $m-1$ copies of (\ref{eq:case2.1a:star}), we instead subtract $n-1$ copies of (\ref{eq:case2.1a:star}) yielding
            \begin{align*}
                (m-n)\alpha_{m}+(m-n+1)\alpha_{m+1} - \alpha_{r_1+1} = \alpha'_{n+1}-\alpha'_{r_1+1}.
            \end{align*}
            Then, since $m-n > 0$, the same argument from the $m<n$ case will follow through by appropriately replacing each occurrence of $\alpha'_n$ with $\alpha_m$ and vice versa, $\alpha'_{n+1}$ with $\alpha_{m+1}$ and vice versa, and $\alpha'_{r_1+1}$ with $\alpha_{r_1+1}$ and vice versa. 
            \item $m = r_1$.
            Since $z_{r_1}y_1\cdots y_{r_1-1}$ (by (\ref{eq:gb2})), $z_{r_1}y_{r_1}\cdots y_d \in \inG$ (by (\ref{eq:gb3})), and $f\notin \inG$, there exist indices $k_2\in \{1,\ldots,r_1-1\}$ and $\ell_2\in \{r_1,\ldots,d\}$ such that $\beta_{k_2} = \beta_{\ell_2} = 0$.
            Then, we have that
            \begin{align}
                \bbf[d-k_2+1] &= \sum_{i=1}^{r_1+3}\alpha_i\sA_{d-k_2+1,i} \label{eq:case2.1b:fk2} \\
                \begin{split} \bbg[d-k_2+1] &= -(1+(r_1-n+1)x_1)\alpha'_n - (1+(r_1-n)x_1)\alpha'_{n+1} \\ &\qquad - x_1\alpha'_{r_1+1} + \beta'_{k_2} \end{split} \label{eq:case2.1b:gk2} \\
                \bbf[d-\ell_2+1] &= \sum_{i=1}^{r_1+3}\alpha_i\sA_{d-\ell_2+1,i} \label{eq:case2.1b:fl2} \\
                \bbg[d-\ell_2+1] &= -(r_1-n+1)\alpha'_n - (r_1-n)\alpha'_{n+1} - \alpha'_{r_1+1} + \beta'_{\ell_2}, \label{eq:case2.1b:gl2}
            \end{align}
            where $(\ref{eq:case2.1b:fk2}) = (\ref{eq:case2.1b:gk2})$ and $(\ref{eq:case2.1b:fl2}) = (\ref{eq:case2.1b:gl2})$ as $\pi(f)=\pi(g)$. 
            Subtracting the equation $(\ref{eq:case2.1b:fk2}) = (\ref{eq:case2.1b:gk2})$ from $(\ref{eq:case2.1:fk1}) = (\ref{eq:case2.1:gk1})$ implies $\beta_{k_1} = -\beta'_{k_2}$.
            Similarly, subtracting equation $(\ref{eq:case2.1b:fl2}) = (\ref{eq:case2.1b:gl2})$ from $(\ref{eq:case2.1:fl1}) = (\ref{eq:case2.1:gl1})$ implies $\beta_{\ell_1}=-\beta'_{\ell_2}$.
            Since $\beta_j,\beta'_j\geq 0$ for all $1\leq j \leq d$, this implies $\beta_{k_1} = \beta'_{k_2} = \beta_{\ell_1} = \beta'_{\ell_2} = 0$.
            Also, by Lemma \ref{lem:3support}, we know $\alpha_{r_1}>0$, $\alpha_{r_1+1},\alpha_{r_1+2}\geq 0$, and $\alpha_i= 0$ for all $i\in \{1,\ldots,r_1-1\}\cup \{r_1+3\}$.
            Consequently, equations (\ref{eq:case2.1:fk1}) and (\ref{eq:case2.1:fl1}) simplify to
            \begin{align}
                \bbf[d-k_1+1] &= -(1+x_1)\alpha_{r_1} - x_1\alpha_{r_1+1} - \alpha_{r_1+2} \label{eq:case2.1b:fk1} \\
                \bbf[d-\ell_1+1] &=  -\alpha_{r_1} - \alpha_{r_1+1}. \label{eq:case2.1b:fl1} 
            \end{align}
            Since $\pi(f)=\pi(g)$, $(\ref{eq:case2.1b:fk1}) = (\ref{eq:case2.1:gk1})$ and $(\ref{eq:case2.1b:fl1}) = (\ref{eq:case2.1:gl1})$, thereby implying $x_1(\ref{eq:case2.1:gl1}) - (\ref{eq:case2.1:gk1}) = x_1(\ref{eq:case2.1b:fl1}) - (\ref{eq:case2.1b:fk1})$.
            Observe that $x_1(\ref{eq:case2.1:gl1}) - (\ref{eq:case2.1:gk1}) = x_1(\ref{eq:case2.1b:fl1}) - (\ref{eq:case2.1b:fk1})$ is the following
            \begin{align}
                \alpha_{r_1}+\alpha_{r_1+2} = \alpha'_n + \alpha'_{n+1}. \label{eq:case2.1b:star}
            \end{align}
            Now, consider the equation $-(\ref{eq:case2.1b:fl1}) = -(\ref{eq:case2.1:gl1})$:
            \begin{align*}
                \alpha_{r_1} + \alpha_{r_1+1} = (r_1-n+1)\alpha'_n + (r_1-n)\alpha'_{n+1} + \alpha'_{r_1+1}.
            \end{align*}
            Substituting (\ref{eq:case2.1b:star}) into this equation yields
            \begin{align*}
                \alpha'_n + \alpha'_{n+1} - \alpha_{r_1+2} + \alpha_{r_1+1} = (r_1-n+1)\alpha'_n + (r_1-n)\alpha'_{n+1} + \alpha'_{r_1+1},
            \end{align*}
            which rearranged is
            \begin{align}
                \alpha_{r_1+1} - \alpha_{r_1+2} = \underbrace{(r_1-n)\alpha'_{n}}_{>\, 0} + \underbrace{(r_1-n-1)\alpha'_{n+1}}_{\geq\, 0} + \alpha'_{r_1+1} \label{eq:case2.1b:2star}.
            \end{align}
            Observe that (\ref{eq:case2.1b:2star}) implies $\alpha_{r_1+1} > 0$, so since $f$ and $g$ are relatively prime, this forces $\alpha'_{r_1+1} = 0$.
            Therefore, subtracting $r_1 - n$ copies of (\ref{eq:case2.1b:star}) from (\ref{eq:case2.1b:2star}) gives
            \begin{align*}
                \alpha_{r_1+1} - (r_1-n)\alpha_{r_1} - (r_1-n+1)\alpha_{r_1+2} = -\alpha'_{n+1},
            \end{align*}
            which implies
            \begin{align}
                \alpha'_{n+1} = (r_1-n)\alpha_{r_1} + (r_1-n+1)\alpha_{r_1+2} - \alpha_{r_1+1}. \label{eq:case2.1b:inside}
            \end{align} 
            Now, $\bbf[d+1] = \bbg[d+1]$ gives that
            \begin{align*}
                \alpha_{r_1} + \alpha_{r_1+1} + \alpha_{r_1+2} + \sum_j \beta_j = \alpha'_n + \alpha'_{n+1} + \sum_j \beta'_j.
            \end{align*}
            Since $\alpha'_n = \alpha_{r_1} + \alpha_{r_1+2}-\alpha'_{n+1}$ by (\ref{eq:case2.1b:star}) and $\alpha_{r_1+1}>0$, this implies $\sum_j \beta_j < \sum_j \beta'_j$. 
            For each $r_1\leq j \leq d$, $-\bbf[d-j+1] = -\bbg[d-j+1]$ is given by
            \begin{align*}
                \alpha_{r_1} + \alpha_{r_1+1} - \beta_j = (r_1-n+1)\alpha'_n + (r_1-n)\alpha'_{n+1} -\beta'_j.
            \end{align*}
            Solving for $\alpha'_{n+1}$ and substituting $\alpha'_n = \alpha_{r_1} + \alpha_{r_1+2}-\alpha'_{n+1}$ yields
            \begin{align*}
                \alpha'_{n+1} = (r_1-n)\alpha_{r_1} + (r_1-n+1)\alpha_{r_1+2} - \alpha_{r_1+1}+\beta_j - \beta'_j.
            \end{align*}
            Adding these equations for each $r_1\leq j \leq d$ gives
            \begin{align}
                \begin{split} (d-r_1+1)\alpha'_{n+1}  &= (d-r_1+1)\left[(r_1-n)\alpha_{r_1} + (r_1-n+1)\alpha_{r_1+2} - \alpha_{r_1+1}\right] \\  &\qquad +\sum_{r_1\leq j\leq d}(\beta_j - \beta'_j). \end{split} \label{eq:case2.1b:dag}
            \end{align}
            Similarly, for each $1\leq j \leq r_1-1$, $-\bbf[d-j+1] = -\bbg[d-j+1]$ is given by
             \begin{align*}
                (1+x_1)\alpha_{r_1} + x_1\alpha_{r_1+1} + \alpha_{r_1+2} - \beta_j = (1+(r_1-n+1)x_1)\alpha'_n + (1+(r_1-n)x_1)\alpha'_{n+1} -\beta'_j.
            \end{align*}
            Solving for $x_1\alpha'_{n+1}$ and making the appropriate substitutions yields
            \begin{align*}
                x_1\alpha'_{n+1} = (r_1-n)x_1\alpha_{r_1} + (r_1-n+1)x_1\alpha_{r_1+2} -x_1\alpha_{r_1+1} +\beta_j - \beta'_j.
            \end{align*}
            Adding these equations for each $1\leq j \leq r_1-1$ gives
            \begin{align}
                \begin{split} (r_1-1)x_1\alpha'_{n+1} &= (r_1-1)\left[(r_1-n)x_1\alpha_{r_1} + (r_1-n+1)x_1\alpha_{r_1+2} - x_1\alpha_{r_1+1}\right] \\ &\qquad +\sum_{1\leq j\leq r_1-1}(\beta_j - \beta'_j). \end{split} \label{eq:case2.1b:ddag}
            \end{align}
            Combining (\ref{eq:case2.1b:dag}) and (\ref{eq:case2.1b:ddag}) gives
            \begin{align*}
                r_1x_1\alpha'_{n+1} = r_1x_1\underbrace{\left[(r_1-n)\alpha_{r_1} + (r_1-n+1)\alpha_{r_1+2} - \alpha_{r_1+1}\right]}_{=\, \alpha'_{n+1} \text{ by } (\ref{eq:case2.1b:inside})} + \underbrace{\sum_{j=1}^{d} (\beta_j - \beta'_j)}_{<\, 0},
            \end{align*}
            a contradiction.
            \item $m = r_1 + 1$. 
            Since $z_{r_1+1}y_{r_1}\cdots y_d \in \inG$ (by (\ref{eq:gb5})) and $f\notin \inG$, there exists an index $\ell_2\in \{r_1,\ldots,d\}$ such that $\beta_{\ell_2} = 0$.
            Then, we have that
            \begin{align}
                \bbf[d-\ell_2+1] &= \sum_{i=1}^{r_1+3}\alpha_i\sA_{d-\ell_2+1,i} \label{eq:case2.1c:fl2} \\
                \bbg[d-\ell_2+1] &= -(r_1-n+1)\alpha'_n - (r_1-n)\alpha'_{n+1} - \alpha'_{r_1+1} + \beta'_{\ell_2}, \label{eq:case2.1c:gl2}
            \end{align}
            where $(\ref{eq:case2.1c:fl2}) = (\ref{eq:case2.1c:gl2})$ as $\pi(f)=\pi(g)$. 
            Subtracting the equation $(\ref{eq:case2.1c:fl2}) = (\ref{eq:case2.1c:gl2})$ from $(\ref{eq:case2.1:fl1}) = (\ref{eq:case2.1:gl1})$ implies $\beta_{\ell_1}=-\beta'_{\ell_2}$.
            Since $\beta_j,\beta'_j\geq 0$ for all $1\leq j \leq d$, this implies $\beta_{\ell_1} = \beta'_{\ell_2} = 0$.
            Also, by Lemma \ref{lem:3support}, we know $\alpha_{r_1+1}>0$, $\alpha_{r_1+2},\alpha_{r_1+3}\geq 0$, and $\alpha_i= 0$ for all $i\in \{1,\ldots,r_1\}$.
            Consequently, since $\beta_{\ell_1} = 0$, equations (\ref{eq:case2.1:fk1}) and (\ref{eq:case2.1:fl1}) simplify to
            \begin{align}
                \bbf[d-k_1+1] &= -x_1\alpha_{r_1+1} - \alpha_{r_1+2} + \beta_{k_1} \label{eq:case2.1c:fk1} \\
                \bbf[d-\ell_1+1] &=  - \alpha_{r_1+1}. \label{eq:case2.1c:fl1} 
            \end{align}
            Furthermore, since $f$ and $g$ are relatively prime, $\alpha_{r_1+1}>0$ implies $\alpha'_{r_1+1} = 0$, so equations (\ref{eq:case2.1:gk1}) and (\ref{eq:case2.1:gl1}) simplify to
            \begin{align}
                \bbg[d-k_1+1] &= -(1+(r_1-n+1)x_1)\alpha'_n - (1+(r_1-n)x_1)\alpha'_{n+1} \label{eq:case2.1c:gk1} \\
                \bbg[d-\ell_1+1] &= -(r_1-n+1)\alpha'_n - (r_1-n)\alpha'_{n+1}. \label{eq:case2.1c:gl1}
            \end{align}
            Since $\pi(f)=\pi(g)$, $(\ref{eq:case2.1c:fk1}) = (\ref{eq:case2.1c:gk1})$ and $(\ref{eq:case2.1c:fl1}) = (\ref{eq:case2.1c:gl1})$.
            Therefore, we have that $-(\ref{eq:case2.1c:fk1}) = -(\ref{eq:case2.1c:gk1})$ and $-(\ref{eq:case2.1c:fl1}) = -(\ref{eq:case2.1c:gl1})$, that is,
            \begin{align}
                x_1\alpha_{r_1+1} + \alpha_{r_1+2} - \beta_{k_1} = (1+(r_1-n+1)x_1)\alpha'_n + (1+(r_1-n)x_1)\alpha'_{n+1} \label{eq:case2.1c:inside1}
            \end{align}
            and
            \begin{align}
                \alpha_{r_1+1} = (r_1-n+1)\alpha'_n + (r_1-n)\alpha'_{n+1}. \label{eq:case2.1c:inside2}
            \end{align}
            Now, $\bbf[d+1] = \bbg[d+1]$ gives that
            \begin{align*}
                \alpha_{r_1+1} + \alpha_{r_1+2} + \alpha_{r_1+3} + \sum_j \beta_j = \alpha'_n + \alpha'_{n+1} + \sum_j \beta'_j.
            \end{align*}
            Substituting \eqref{eq:case2.1c:inside2} and since $(r_1-n)\alpha'_{n}>0$, we obtain 
            \begin{equation} \label{eq:betaineq}
            \sum_j \beta_j < \sum_j \beta'_j \, .
            \end{equation}
            For each $r_1\leq j \leq d$, $-\bbf[d-j+1] = -\bbg[d-j+1]$ is given by
            \begin{align*}
                \alpha_{r_1+1} - \beta_j = (r_1-n+1)\alpha'_n + (r_1-n)\alpha'_{n+1} -\beta'_j,
            \end{align*}
            which readily implies
            \begin{align*}
                \alpha_{r_1+1} = (r_1-n+1)\alpha'_n + (r_1-n)\alpha'_{n+1} + \beta_j - \beta'_j.
            \end{align*}
            Adding these equations for each $r_1\leq j \leq d$ gives
            \begin{align*}
                \begin{split} (d-r_1+1)\alpha_{r_1+1}  &= (d-r_1+1)\left[(r_1-n+1)\alpha'_{n} + (r_1-n)\alpha'_{n+1}\right] \\  &\qquad + \sum_{r_1\leq j\leq d}(\beta_j - \beta'_j). \end{split} 
            \end{align*}
            Using (\ref{eq:case2.1c:inside2}), this simplifies to
            \begin{align}
                0  &=  \sum_{r_1\leq j\leq d}(\beta_j - \beta'_j) \, .
                \label{eq:case2.1c:dag}
            \end{align}
            Similarly, for each $1\leq j \leq r_1-1$, $-\bbf[d-j+1] = -\bbg[d-j+1]$ is given by
            \begin{align*}
                x_1\alpha_{r_1+1} + \alpha_{r_1+2} - \beta_j = (1+(r_1-n+1)x_1)\alpha'_n + (1+(r_1-n)x_1)\alpha'_{n+1} - \beta'_j,
            \end{align*}
            which implies
            \begin{align*}
                x_1\alpha_{r_1+1} = (1+(r_1-n+1)x_1)\alpha'_n + (1+(r_1-n)x_1)\alpha'_{n+1} -\alpha_{r_1+2}  + \beta_j - \beta'_j.
            \end{align*}
            Adding these equations for each $1\leq j \leq r_1-1$ gives
            \begin{align*}
                \begin{split} (r_1-1)x_1\alpha_{r_1+1} &= (r_1-1)\big[(1+(r_1-n+1)x_1)\alpha'_n + (1+(r_1-n)x_1)\alpha'_{n+1}  \\ &\qquad - \alpha_{r_1+2} \big] +\sum_{1\leq j\leq r_1-1}(\beta_j - \beta'_j) \, . \end{split} 
            \end{align*}
            Using (\ref{eq:case2.1c:inside1}), this simplifies to
            \begin{align}
                0 &= -(r_1-1) \beta_{k_1} + \sum_{1\leq j\leq r_1-1}(\beta_j - \beta'_j).
                \label{eq:case2.1c:ddag}
            \end{align}
            Combining (\ref{eq:case2.1c:dag}) and (\ref{eq:case2.1c:ddag}), and observing \eqref{eq:betaineq}, gives
            \begin{align*}
                0 = -(r_1-1)\beta_{k_1} + \underbrace{\sum_{j=1}^{d} (\beta_j - \beta'_j)}_{<\, 0},
            \end{align*}
            which implies $(r_1-1)\beta_{k_1} < 0$, a contradiction.
        \end{enumerate}
    \end{subcase}
    \begin{subcase}$n=r_1$.
        By Lemma $\ref{lem:3support}$, we know $\alpha'_{r_1} > 0$, $\alpha'_{r_1+1},\alpha'_{r_1+2} \geq 0$, and $\alpha'_i = 0$ for all $i\in \{1,\ldots,r_1-1\} \cup \{r_1+3\}$. 
        Since $z_{r_1}y_1\cdots y_{r_1-1}\in \inG$ (by (\ref{eq:gb2})), $z_{r_1}y_{r_1}\cdots y_{d}\in \inG$ (by (\ref{eq:gb3})), and $g\notin \inG$, there exist indices $k_1\in \{1,\ldots, r_1-1\}$ and $\ell_1\in \{r_1,\ldots, d\}$ such that $\beta'_{k_1} = \beta'_{\ell_1} = 0$. 
        Then,
        \begin{align}
            \bbf[d-k_1+1] &= \sum_{i=1}^{r_1+3}\alpha_i\sA_{d-k_1+1,i} + \beta_{k_1} \label{eq:case2.2:fk1} \\
            \bbg[d-k_1+1] &= -(1+x_1)\alpha'_{r_1} - x_1\alpha'_{r_1+1} - \alpha'_{r_1+2} \label{eq:case2.2:gk1} \\
            \bbf[d-\ell_1+1] &= \sum_{i=1}^{r_1+3}\alpha_i\sA_{d-\ell_1+1,i} + \beta_{\ell_1} \label{eq:case2.2:fl1} \\
            \bbg[d-\ell_1+1] &= -\alpha'_{r_1} - \alpha'_{r_1+1}. \label{eq:case2.2:gl1}
        \end{align}
        Note that $\pi(f)=\pi(g)$ implies $(\ref{eq:case2.2:fk1}) = (\ref{eq:case2.2:gk1})$ and $(\ref{eq:case2.2:fl1}) = (\ref{eq:case2.2:gl1})$.
        Now, we claim $m\in \{1,\ldots,r_1-1\} \cup \{r_1+1\}$ (we need not consider $m=r_1$ since $f$ and $g$ are relatively prime and $n=r_1$ in this case). 
        Indeed, assume otherwise, that is, $\zsupp{f} \subseteq \{r_1+2,r_1+3\}$. 
        Then, $\bbf[d-\ell+1] = \beta_{\ell} \geq 0$ for all $\ell\in \{r_1,\ldots, d\}$, but from (\ref{eq:case2.2:gl1}), $\bbg[d-\ell_1+1] < 0$ since $\alpha'_{r_1} >0$ and $\alpha'_{r_1+1}\geq 0$. 
        This contradicts $\pi(f) = \pi(g)$. 
        Hence, given the structure of Lemma \ref{lem:3support} and since $m$ cannot be $r_1$, we consider the following subsubcases.
        \begin{enumerate}[label={\textbf{(\alph*)}},leftmargin=.575in]
            \item $m\in \{1,\ldots,r_1-1\}$. 
            Since $z_my_1\cdots y_{r_1-1}$ (by (\ref{eq:gb2})), $z_my_{r_1}\cdots y_d \in \inG$ (by (\ref{eq:gb3})), and $f\notin \inG$, there exist indices $k_2\in \{1,\ldots,r_1-1\}$ and $\ell_2\in \{r_1,\ldots,d\}$ such that $\beta_{k_2} = \beta_{\ell_2} = 0$.
            Then, we have that 
            \begin{align}
                \bbf[d-k_2+1] &= \sum_{i=1}^{r_1+3}\alpha_i\sA_{d-k_2+1,i} \label{eq:case2.2a:fk2} \\
                \bbg[d-k_2+1] &= -(1+x_1)\alpha'_{r_1} - x_1\alpha'_{r_1+1} - \alpha'_{r_1+2} + \beta'_{k_2} \label{eq:case2.2a:gk2} \\
                \bbf[d-\ell_2+1] &= \sum_{i=1}^{r_1+3}\alpha_i\sA_{d-\ell_2+1,i} \label{eq:case2.2a:fl2} \\
                \bbg[d-\ell_2+1] &= -\alpha'_{r_1} - \alpha'_{r_1+1} + \beta'_{\ell_2}, \label{eq:case2.2a:gl2}
            \end{align}
            where $(\ref{eq:case2.2a:fk2}) = (\ref{eq:case2.2a:gk2})$ and $(\ref{eq:case2.2a:fl2}) = (\ref{eq:case2.2a:gl2})$ as $\pi(f)=\pi(g)$. 
            Subtracting the equation $(\ref{eq:case2.2a:fk2}) = (\ref{eq:case2.2a:gk2})$ from $(\ref{eq:case2.2:fk1}) = (\ref{eq:case2.2:gk1})$ implies $\beta_{k_1} = -\beta'_{k_2}$.
            Similarly, subtracting equation $(\ref{eq:case2.2a:fl2}) = (\ref{eq:case2.2a:gl2})$ from $(\ref{eq:case2.2:fl1}) = (\ref{eq:case2.2:gl1})$ implies $\beta_{\ell_1}=-\beta'_{\ell_2}$.
            Since $\beta_j,\beta'_j\geq 0$ for all $1\leq j \leq d$, this implies $\beta_{k_1} = \beta'_{k_2} = \beta_{\ell_1} = \beta'_{\ell_2} = 0$. 
            Also, by Lemma \ref{lem:3support}, we know $\alpha_m>0$, $\alpha_{m+1},\alpha_{r_1+1}\geq 0$, and $\alpha_i= 0$ for all $i\in \{1,\ldots,r_1+3\}\setminus \{m,m+1,r_1+1\}$.
            Consequently, equations (\ref{eq:case2.2:fk1}) and (\ref{eq:case2.2:fl1}) simplify to
            \begin{align}
                \bbf[d-k_1+1] &= -(1+(r_1-m+1)x_1)\alpha_m - (1+(r_1-m)x_1)\alpha_{m+1} - x_1\alpha_{r_1+1} \label{eq:case2.2a:fk1} \\
                \bbf[d-\ell_1+1] &=  -(r_1-m+1)\alpha_m - (r_1-m)\alpha_{m+1} - \alpha_{r_1+1}. \label{eq:case2.2a:fl1} 
            \end{align}
            Since $\pi(f)=\pi(g)$, $(\ref{eq:case2.2a:fk1}) = (\ref{eq:case2.2:gk1})$ and $(\ref{eq:case2.2a:fl1}) = (\ref{eq:case2.2:gl1})$, thereby implying $x_1(\ref{eq:case2.2:gl1}) - (\ref{eq:case2.2:gk1}) = x_1(\ref{eq:case2.2a:fl1}) - (\ref{eq:case2.2a:fk1})$.
            Observe that $x_1(\ref{eq:case2.2:gl1}) - (\ref{eq:case2.2:gk1}) = x_1(\ref{eq:case2.2a:fl1}) - (\ref{eq:case2.2a:fk1})$ is the following
            \begin{align}
                \alpha_m+\alpha_{m+1} = \alpha'_{r_1} + \alpha'_{r_1+2}. \label{eq:case2.2a:star}
            \end{align}
            Now, consider the equation $-(\ref{eq:case2.2a:fl1}) = -(\ref{eq:case2.2:gl1})$:
            \begin{align*}
                (r_1-m+1)\alpha_m + (r_1-m)\alpha_{m+1} + \alpha_{r_1+1} = \alpha'_{r_1} + \alpha'_{r_1+1}.
            \end{align*}
            Substituting (\ref{eq:case2.2a:star}) into this equation and solving for $\alpha'_{r_1+1}$ yields
            \begin{align}
                \alpha'_{r_1+1} = (r_1-m)\alpha_{m} + (r_1-m-1)\alpha_{m+1} + \alpha_{r_1+1} + \alpha'_{r_1+2}. \label{eq:case2.2a:2star}
            \end{align}
            Observe that (\ref{eq:case2.2a:2star}) implies $\alpha'_{r_1+1} > 0$, so since $f$ and $g$ are relatively prime, this forces $\alpha_{r_1+1} = 0$.
            Thus, $\zsupp{f}\subseteq \{m,m+1\}$.
            Moreover, since $\abs{\zsupp{f}}\geq \abs{\zsupp{g}}$, $\alpha_{r_1+1}=0$, and we have $\alpha'_{r_1},\alpha'_{r_1+1} > 0$, it follows that $\alpha_{m+1}>0$ and $\alpha'_{r_1+2} = 0$. 
            Consequently, (\ref{eq:case2.2a:star}) reduces to $\alpha'_{r_1} = \alpha_m + \alpha_{m+1}$ and (\ref{eq:case2.2a:2star}) reduces to 
            \begin{align*}
                \alpha'_{r_1+1} = (r_1-m)\alpha_m + (r_1-m-1)\alpha_{m+1}.
            \end{align*}
            Summing these reduced equations yields 
            \begin{align}
                \alpha'_{r_1}+\alpha'_{r_1+1} = (r_1-m+1)\alpha_m + (r_1-m)\alpha_{m+1}. \label{eq:case2.2a:stars}
            \end{align}
            Now, $\bbf[d+1] = \bbg[d+1]$ gives that
            \begin{align*}
                \alpha_m + \alpha_{m+1} + \sum_j \beta_j = \alpha'_{r_1} + \alpha'_{r_1+1} + \sum_j \beta'_j.
            \end{align*}
            Since $\alpha'_{r_1} = \alpha_m + \alpha_{m+1}$ and $\alpha'_{r_1+1}>0$, this implies 
            \begin{equation}\label{eq:betasubcase2.2a}
                \sum_j \beta_j > \sum_j \beta'_j \, .
            \end{equation} 
            For each $r_1\leq j \leq d$, $-\bbf[d-j+1] = -\bbg[d-j+1]$ is given by
            \begin{align*}
                (r_1-m+1)\alpha_m + (r_1-m)\alpha_{m+1} - \beta_j = \alpha'_{r_1} + \alpha'_{r_1+1} -\beta'_j,
            \end{align*}
            which, via (\ref{eq:case2.2a:stars}), implies $\beta_j = \beta'_j$.
            Similarly, for each $1\leq j \leq r_1-1$, $-\bbf[d-j+1] = -\bbg[d-j+1]$ is given by
             \begin{align*}
                (1+(r_1-m+1)x_1)\alpha_m + (1+(r_1-m)x_1)\alpha_{m+1} - \beta_j = (1+x_1)\alpha'_{r_1} + x_1\alpha'_{r_1+1} -\beta'_j,
            \end{align*}
            which, via (\ref{eq:case2.2a:stars}), implies $\beta_j = \beta'_j$.
            Thus, we have that $\beta_j = \beta'_j$ for all $1\leq j \leq d$, but we had in \eqref{eq:betasubcase2.2a} that $\sum_j \beta_j > \sum_j \beta'_j$, a contradiction.
            \item $m = r_1 + 1$. 
            Since $z_{r_1+1}y_{r_1}\cdots y_d \in \inG$ (by (\ref{eq:gb5})) and $f\notin \inG$, there exists an index $\ell_2\in \{r_1,\ldots,d\}$ such that $\beta_{\ell_2} = 0$.
            Then, we have that
            \begin{align}
                \bbf[d-\ell_2+1] &= \sum_{i=1}^{r_1+3}\alpha_i\sA_{d-\ell_2+1,i} \label{eq:case2.2b:fl2} \\
                \bbg[d-\ell_2+1] &= -\alpha'_{r_1} - \alpha'_{r_1+1} + \beta'_{\ell_2}, \label{eq:case2.2b:gl2}
            \end{align}
            where $(\ref{eq:case2.2b:fl2}) = (\ref{eq:case2.2b:gl2})$ as $\pi(f)=\pi(g)$. 
            Subtracting the equation $(\ref{eq:case2.2b:fl2}) = (\ref{eq:case2.2b:gl2})$ from $(\ref{eq:case2.2:fl1}) = (\ref{eq:case2.2:gl1})$ implies $\beta_{\ell_1}=-\beta'_{\ell_2}$.
            Since $\beta_j,\beta'_j\geq 0$ for all $1\leq j \leq d$, this implies $\beta_{\ell_1} = \beta'_{\ell_2} = 0$.
            Also, by Lemma \ref{lem:3support}, we know $\alpha_{r_1+1}>0$, $\alpha_{r_1+2},\alpha_{r_1+3}\geq 0$, and $\alpha_i= 0$ for all $i\in \{1,\ldots,r_1\}$.
            Consequently, since $\beta_{\ell_1} = 0$, equations (\ref{eq:case2.2:fk1}) and (\ref{eq:case2.2:fl1}) simplify to
            \begin{align}
                \bbf[d-k_1+1] &= -x_1\alpha_{r_1+1} - \alpha_{r_1+2} + \beta_{k_1} \label{eq:case2.2b:fk1} \\
                \bbf[d-\ell_1+1] &=  - \alpha_{r_1+1}. \label{eq:case2.2b:fl1} 
            \end{align}
            Furthermore, since $f$ and $g$ are relatively prime, $\alpha_{r_1+1}>0$ implies $\alpha'_{r_1+1} = 0$, so equations (\ref{eq:case2.2:gk1}) and (\ref{eq:case2.2:gl1}) simplify to
            \begin{align}
                \bbg[d-k_1+1] &= -(1+x_1)\alpha'_{r_1} - \alpha'_{r_1+2} \label{eq:case2.2b:gk1} \\
                \bbg[d-\ell_1+1] &= -\alpha'_{r_1}. \label{eq:case2.2b:gl1}
            \end{align}
            Since $\pi(f)=\pi(g)$, $(\ref{eq:case2.2b:fk1}) = (\ref{eq:case2.2b:gk1})$ and $(\ref{eq:case2.2b:fl1}) = (\ref{eq:case2.2b:gl1})$.
            Therefore, we have that $-(\ref{eq:case2.2b:fk1}) = -(\ref{eq:case2.2b:gk1})$ and $-(\ref{eq:case2.2b:fl1}) = -(\ref{eq:case2.2b:gl1})$, that is,
            \begin{align}
                (1+x_1)\alpha'_{r_1} + \alpha'_{r_1+2} = x_1\alpha_{r_1+1} + \alpha_{r_1+2} - \beta_{k_1} \label{eq:case2.2b:2star}
            \end{align}
            and
            \begin{align}
                \alpha'_{r_1} = \alpha_{r_1+1}. \label{eq:case2.2b:star}
            \end{align}
            Plugging (\ref{eq:case2.2b:star}) into (\ref{eq:case2.2b:2star}) and solving for $\beta_{k_1}$ gives
            \begin{align}
                \beta_{k_1} = \alpha_{r_1+2}-\alpha_{r_1+1} - \alpha'_{r_1+2}. \label{eq:case2.2b:final}
            \end{align}
            Note that if $\alpha'_{r_1+2}>0$, the relatively prime condition would force $\alpha_{r_1+2} = 0$, thereby implying $\beta_{k_1}<0$, a contradiction.
            Hence, we may assume $\alpha'_{r_1+2} = 0$, and since $\beta_{k_1}\geq 0$, it must be that $\alpha_{r_1+2}>0$. 
            Now, $\bbf[d+1] = \bbg[d+1]$ gives that
            \begin{align*}
                \alpha_{r_1+1} + \alpha_{r_1+2} + \alpha_{r_1+3} + \sum_j \beta_j = \alpha'_{r_1} + \sum_j \beta'_j.
            \end{align*}
            Substituting (\ref{eq:case2.2b:star}) and since $\alpha_{r_1+2}>0$, this implies 
            \begin{equation} \label{eq:betaineqcase2.2b}
            \sum_j \beta_j < \sum_j \beta'_j \, .
            \end{equation}
            For each $r_1\leq j \leq d$, $-\bbf[d-j+1] = -\bbg[d-j+1]$ is given by
            \begin{align*}
                \alpha_{r_1+1} - \beta_j =  \alpha'_{r_1} -\beta'_j,
            \end{align*}
            which, via (\ref{eq:case2.2b:star}), implies $\beta_j = \beta'_j$.
            Similarly, for each $1\leq j \leq r_1-1$, $-\bbf[d-j+1] = -\bbg[d-j+1]$ is given by
             \begin{align*}
                x_1\alpha_{r_1+1} + \alpha_{r_1+2} - \beta_j = (1+x_1)\alpha'_{r_1} - \beta'_j,
            \end{align*}
            which, via (\ref{eq:case2.2b:2star}), implies $\beta_{k_1} = \beta_j - \beta'_j$.
            Therefore,
            \begin{align*}
                0< \sum_{j=1}^{d}(\beta'_j-\beta_j) 
                =
                \sum_{j=1}^{r_1-1}(\beta'_j-\beta_j) 
                +
                \sum_{j=r_1}^{d}(\beta'_j-\beta_j) 
                =
               -(r_1-1)\beta_{k_1}
               \leq 0 \, ,
            \end{align*}
            a contradiction. 
        \end{enumerate}
    \end{subcase}
    \begin{subcase}$n=r_1+1$.
    By Lemma $\ref{lem:3support}$, we know $\alpha'_{r_1+1} > 0$, $\alpha'_{r_1+2},\alpha'_{r_1+3} \geq 0$, and $\alpha'_i = 0$ for all $i\in \{1,\ldots,r_1\}$. 
        Since $z_{r_1+1}y_{r_1}\cdots y_{d}\in \inG$ (by (\ref{eq:gb5})) and $g\notin \inG$, there exists an index $\ell_1\in \{r_1,\ldots, d\}$ such that $\beta'_{\ell_1} = 0$. 
        Then,
        \begin{align}
            \bbf[d-\ell_1+1] &= \sum_{i=1}^{r_1+3}\alpha_i\sA_{d-\ell_1+1,i} + \beta_{\ell_1} \label{eq:case2.3:fl1} \\
            \bbg[d-\ell_1+1] &= -\alpha'_{r_1+1}\label{eq:case2.3:gl1},
        \end{align}
        where $(\ref{eq:case2.3:fl1}) = (\ref{eq:case2.3:gl1})$ as $\pi(f)=\pi(g)$.
        Now, we claim $m\in \{1,\ldots,r_1\}$ (we need not consider $m=r_1+1$ since $f$ and $g$ are relatively prime and $n=r_1+1$ in this case). 
        Indeed, assume otherwise, that is, $\zsupp{f} \subseteq \{r_1+2,r_1+3\}$. 
        Then, $\bbf[d-\ell+1] = \beta_{\ell} \geq 0$ for all $\ell\in \{r_1,\ldots, d\}$, but from (\ref{eq:case2.3:gl1}), $\bbg[d-\ell_1+1] < 0$ since $\alpha'_{r_1+1} > 0$. 
        This contradicts $\pi(f) = \pi(g)$. 
        Hence, given the structure of Lemma \ref{lem:3support} and since $m$ cannot be $r_1+1$, we consider the following subsubcases.
        \begin{enumerate}[label={\textbf{(\alph*)}},leftmargin=.575in]
            \item $m\in \{1,\ldots,r_1-1\}$.
            Since $z_my_{1}\cdots y_{r_1-1} \in \inG$ (by (\ref{eq:gb2})), $z_my_{r_1}\cdots y_d \in \inG$ (by (\ref{eq:gb3})), and $f\notin \inG$, there exist indices $k_2\in \{1,\ldots,r_1-1\}$ and $\ell_2\in \{r_1,\ldots,d\}$ such that $\beta_{k_2} = \beta_{\ell_2} = 0$.
            Then, we have that 
            \begin{align}
                \bbf[d-k_2+1] &= \sum_{i=1}^{r_1+3}\alpha_i\sA_{d-k_2+1,i} \label{eq:case2.3a:fk2} \\
                \bbg[d-k_2+1] &= -x_1\alpha'_{r_1+1} - \alpha'_{r_1+2} + \beta'_{k_2} \label{eq:case2.3a:gk2} \\
                \bbf[d-\ell_2+1] &= \sum_{i=1}^{r_1+3}\alpha_i\sA_{d-\ell_2+1,i} \label{eq:case2.3a:fl2} \\
                \bbg[d-\ell_2+1] &= -\alpha'_{r_1+1} + \beta'_{\ell_2}, \label{eq:case2.3a:gl2}
            \end{align}
            where $(\ref{eq:case2.3a:fk2}) = (\ref{eq:case2.3a:gk2})$ and $(\ref{eq:case2.3a:fl2}) = (\ref{eq:case2.3a:gl2})$ since $\pi(f)=\pi(g)$. 
            Subtracting the equation $(\ref{eq:case2.3a:fl2}) = (\ref{eq:case2.3a:gl2})$ from $(\ref{eq:case2.3:fl1}) = (\ref{eq:case2.3:gl1})$ implies $\beta_{\ell_1} = -\beta'_{\ell_2}$. Since $\beta_j,\beta'_j\geq 0$ for all $1\leq j \leq d$, this implies $\beta_{\ell_1} = \beta'_{\ell_2} = 0$.
            Also, by Lemma \ref{lem:3support} and since $n=r_1+1$, we know $\alpha_m>0$, $\alpha_{m+1}\geq 0$, $\alpha_{r_1+1}=0$, and $\alpha_i= 0$ for all $i\in \{1,\ldots,r_1+3\}\setminus \{m,m+1\}$.
            Consequently, since $\beta_{\ell_1} = 0$, the equation $(\ref{eq:case2.3:fl1}) = (\ref{eq:case2.3:gl1})$ simplifies to
            \begin{align*}
                 -(r_1-m+1)\alpha_m - (r_1-m)\alpha_{m+1} = -\alpha'_{r_1+1},
            \end{align*}
            which implies
            \begin{align}
                \alpha'_{r_1+1} = (r_1-m+1)\alpha_m + (r_1-m)\alpha_{m+1}. \label{eq:case2.3a:star}
            \end{align}
            Furthermore, the equation $-(\ref{eq:case2.3a:fk2}) = -(\ref{eq:case2.3a:gk2})$ simplifies to
            \begin{align*}
                (1+(r_1-m+1)x_1)\alpha_m + (1+(r_1-m)x_1)\alpha_{m+1} = x_1\alpha'_{r_1+1} + \alpha'_{r_1+2} - \beta'_{k_2}.
            \end{align*}
            Via (\ref{eq:case2.3a:star}), this equation is equivalent to
            \begin{align*}
                \alpha_m + \alpha_{m+1} + x_1\alpha'_{r_1+1} = x_1\alpha'_{r_1+1} + \alpha'_{r_1+2} - \beta'_{k_2},
            \end{align*}
            which implies
            \begin{align}
                \beta'_{k_2} = \alpha'_{r_1+2} - \alpha_m - \alpha_{m+1}. \label{eq:case2.3a:2star}
            \end{align}
            Note that if $\alpha'_{r_1+2} = 0$, $\beta'_{k_2} < 0$ by (\ref{eq:case2.3a:2star}), a contradiction. 
            Hence, we may assume $\alpha'_{r_1+2} > 0$.
            Also, since $\abs{\zsupp{f}} \geq \abs{\zsupp{g}}$ and $\alpha_{r_1+1} = 0$, it follows that $\alpha_{m+1}>0$ and $\alpha'_{r_1+3}=0$.
            Now, $\bbf[d+1] = \bbg[d+1]$ gives that
            \begin{align*}
                \alpha_{m} + \alpha_{m+1} + \sum_j \beta_j = \alpha'_{r_1+1} + \alpha'_{r_1+2} + \sum_j \beta'_j.
            \end{align*}
            Substituting (\ref{eq:case2.3a:star}), this implies
            \begin{equation}\label{eq:betaineqcase2.3b}
            \sum_j \beta_j > \sum_j \beta'_j \, .
            \end{equation} 
            For each $r_1\leq j \leq d$, $-\bbf[d-j+1] = -\bbg[d-j+1]$ is given by
            \begin{align*}
                (r_1-m+1)\alpha_{m} + (r_1-m)\alpha_{m+1} - \beta_j =  \alpha'_{r_1+1} -\beta'_j,
            \end{align*}
            which, via (\ref{eq:case2.3a:star}), implies $\beta_j = \beta'_j$.
            Similarly, for each $1\leq j \leq r_1-1$, $-\bbf[d-j+1] = -\bbg[d-j+1]$ is given by
             \begin{align*}
                (1+(r_1-m+1)x_1)\alpha_{m} + (1+(r_1-m)x_1)\alpha_{m+1} - \beta_j = x_1\alpha'_{r_1+1} + \alpha'_{r_1+2} - \beta'_j,
            \end{align*}
            which, via (\ref{eq:case2.3a:star}) and (\ref{eq:case2.3a:2star}), implies 
            \begin{align} 
                \beta'_{k_2} = \beta'_j - \beta_j. \label{eq:case2.3a:final}
            \end{align}
            Therefore, by~\eqref{eq:betaineqcase2.3b} and~\eqref{eq:case2.3a:final},
            \begin{align*}
               0< \sum_{j=1}^{d}(\beta_j-\beta'_j) = \sum_{j=1}^{r_1-1}(\beta_j-\beta'_j) +\sum_{j=r_1}^{d}(\beta_j-\beta'_j) = \sum_{j=1}^{r_1-1}(\beta_j-\beta'_j) = -(r_1-1)\beta'_{k_2} \leq 0,
            \end{align*}
            a contradiction.
            \item $m=r_1$.
            Since $z_{r_1}y_{1}\cdots y_{r_1-1} \in \inG$ (by (\ref{eq:gb2})), $z_{r_1}y_{r_1}\cdots y_d \in \inG$ (by (\ref{eq:gb3})), and $f\notin \inG$, there exist indices $k_2\in \{1,\ldots,r_1-1\}$ and $\ell_2\in \{r_1,\ldots,d\}$ such that $\beta_{k_2} = \beta_{\ell_2} = 0$.
            Then, we have that 
            \begin{align}
                \bbf[d-k_2+1] &= \sum_{i=1}^{r_1+3}\alpha_i\sA_{d-k_2+1,i} \label{eq:case2.3b:fk2} \\
                \bbg[d-k_2+1] &= -x_1\alpha'_{r_1+1} - \alpha'_{r_1+2} + \beta'_{k_2} \label{eq:case2.3b:gk2} \\
                \bbf[d-\ell_2+1] &= \sum_{i=1}^{r_1+3}\alpha_i\sA_{d-\ell_2+1,i} \label{eq:case2.3b:fl2} \\
                \bbg[d-\ell_2+1] &= -\alpha'_{r_1+1} + \beta'_{\ell_2}, \label{eq:case2.3b:gl2}
            \end{align}
            where $(\ref{eq:case2.3b:fk2}) = (\ref{eq:case2.3b:gk2})$ and $(\ref{eq:case2.3b:fl2}) = (\ref{eq:case2.3b:gl2})$ since $\pi(f)=\pi(g)$. 
            Subtracting the equation $(\ref{eq:case2.3b:fl2}) = (\ref{eq:case2.3b:gl2})$ from $(\ref{eq:case2.3:fl1}) = (\ref{eq:case2.3:gl1})$ implies $\beta_{\ell_1} = -\beta'_{\ell_2}$. Since $\beta_j,\beta'_j\geq 0$ for all $1\leq j \leq d$, this implies $\beta_{\ell_1} = \beta'_{\ell_2} = 0$.
            We know $\alpha_{r_1+1} = 0$ since $\alpha'_{r_1+1}>0$. 
            Also, by Lemma \ref{lem:3support}, we know $\alpha_{r_1}>0$ and $\alpha_{r_1+2}\geq 0$, so it follows that $\alpha_i= 0$ for all $i\in \{1,\ldots,r_1+3\}\setminus \{r_1,r_1+2\}$.
            Consequently, since $\beta_{\ell_1} = 0$, the equation $(\ref{eq:case2.3:fl1}) = (\ref{eq:case2.3:gl1})$ simplifies to
            \begin{align}
                \alpha_{r_1} = \alpha'_{r_1+1}. \label{eq:case2.3b:star}
            \end{align}
            Furthermore, the equation $-(\ref{eq:case2.3b:fk2}) = -(\ref{eq:case2.3b:gk2})$ simplifies to
            \begin{align*}
                (1+x_1)\alpha_{r_1} + \alpha_{r_1+2} = x_1\alpha'_{r_1+1} + \alpha'_{r_1+2} - \beta'_{k_2}.
            \end{align*}
            Via (\ref{eq:case2.3b:star}), this equation is equivalent to
            \begin{align*}
                (1+x_1)\alpha'_{r_1+1} + \alpha_{r_1+2} = x_1\alpha'_{r_1+1} + \alpha'_{r_1+2} - \beta'_{k_2},
            \end{align*}
            which implies
            \begin{align}
                \beta'_{k_2} = \alpha'_{r_1+2} - \alpha'_{r_1+1} - \alpha_{r_1+2}. \label{eq:case2.3b:2star}
            \end{align}
            Note that if $\alpha'_{r_1+2} = 0$, $\beta'_{k_2} < 0$ by (\ref{eq:case2.3b:2star}), a contradiction. 
            Hence, it must be that $\alpha'_{r_1+2} > 0$.
            However, by the relatively prime condition, this implies $\alpha_{r_1+2} = 0$.
            As a consequence, since $\alpha_{r_1+1} = \alpha_{r_1+2} = 0$ and $\alpha'_{r_1+1}, \alpha'_{r_1+2} > 0$, we have that
            \begin{align*}
                \abs{\zsupp{f}} = 1 < 2 \leq \abs{\zsupp{g}},
            \end{align*}
            contradicting our assumption that $\abs{\zsupp{f}} \geq \abs{\zsupp{g}}$. 
        \end{enumerate}
    \end{subcase}
    \begin{subcase}$n\in \{r_1+2, r_1+3\}$.
    In this case, $\zsupp{g} \subseteq \{r_1+2,r_1+3\}$.
    Consequently, for $1\leq j \leq d$, we have that
    \begin{align}
        \bbg[d-j+1] = \begin{cases}
            -\alpha'_{r_1+2} + \beta'_{j}, \quad & \text{ for } 1\leq j \leq r_1-1 \\
            \beta'_{j}, \quad & \text{ for } r_1\leq j \leq d.
        \end{cases}
        \label{eq:case2.4}
    \end{align}
    Now, we consider the possibilities for $m$.
        \begin{enumerate}[label={\textbf{(\alph*)}},leftmargin=.575in]
            \item $m\in \{1,\ldots,r_1+1\}$.
            Since $z_my_{r_1}\cdots y_d \in \inG$ (by (\ref{eq:gb3}) or (\ref{eq:gb5})) and $f\notin \inG$, there exists an index $\ell_1\in \{r_1,\ldots,d\}$ such that $\beta_{\ell_1} = 0$.
            Therefore, since $\alpha_{m}>0$, we have that
            \begin{align*}
                \bbf[d-\ell_1+1] &= \underbrace{\sum_{i=1}^{r_1+3}\alpha_i\sA_{d-\ell_1+1,i}}_{<\, 0} + \underbrace{\sum_{j=1}^{d} \beta_j \sA_{d-\ell_1+1,r_1+3+j}}_{=\,0} < 0,
            \end{align*}
            but this contradicts $\pi(f)=\pi(g)$ since $\bbg[d-\ell_1+1] = \beta'_{\ell_1} \geq 0$ from (\ref{eq:case2.4}).
            \item $m\in \{r_1+2,r_1+3\}$.
            Note that since the relatively prime condition implies $m\neq n$, it follows that $\abs{\zsupp{f}} = \abs{\zsupp{g}} = 1$ in this case. 
            Therefore, we may assume without loss of generality that $m = r_1+2$ and $n={r_1+3}$.
            Since $z_{r_1+2}y_1\ldots y_{r_1-1} \in \inG$ (by (\ref{eq:gb4})) and $f\notin \inG$, there exists an index $k_1\in\{1,\ldots,r_1-1\}$ such that $\beta_{k_1} = 0$. 
            Therefore, since $\alpha_{r_1+2} > 0$, we have that $\alpha'_{r_1+2} = 0$ and
            \begin{align*}
                \bbf[d-k_1+1] &= \underbrace{\sum_{i=1}^{r_1+3}\alpha_i\sA_{d-k_1+1,i}}_{<\, 0} + \underbrace{\sum_{j=1}^{d} \beta_j \sA_{d-k_1+1,r_1+3+j}}_{=\,0} < 0.
            \end{align*}
            However, this contradicts $\pi(f)=\pi(g)$ since $\bbg[d-k_1+1] = -\underbrace{\alpha'_{r_1+2}}_{=\, 0} + \beta'_{k_1} \geq 0$ from (\ref{eq:case2.4}). 
        \end{enumerate}
    \end{subcase}
\end{case}
\noindent Since each of the above cases (which together cover all possible pairs $(m,n)$) yields a contradiction, Lemma \ref{lem:grobner} implies that $\sG$ forms a Gr\"{o}bner basis of $I_{\sA}$ with respect to $<_{lex}$, as required.
\end{proof}

In sum, since we have demonstrated that $\sG$ is a Gr\"{o}bner basis of $I_{\sA}$ with respect to $<_{lex}$, we know $in_{<_{lex}}(\sG) = in_{<_{lex}}(I_{\sA})$.
Therefore, since we can clearly see $in_{<_{lex}}(\sG)$ is squarefree, Theorem \ref{thm:main} holds and \cite[Corollary 8.9]{sturmfels} proves Corollary \ref{cor:main}.
As such, there exists a regular unimodular triangulation of the points in $\sA'$, as desired.

\bibliographystyle{plain}
\bibliography{Braun}

\begin{thebibliography}{10}

\bibitem{LaplacianSimplicesDigraphs}
Gabriele Balletti, Takayuki Hibi, Marie Meyer, and Akiyoshi Tsuchiya.
\newblock Laplacian simplices associated to digraphs.
\newblock {\em Arkiv för matematik}, 56, 12 2018.

\bibitem{beck2007computing}
M.~Beck and S.~Robins.
\newblock {\em Computing the Continuous Discretely: Integer-point Enumeration
  in Polyhedra}.
\newblock Undergraduate Texts in Mathematics. Springer New York, 2007.

\bibitem{BraunDavisReflexive}
Benjamin Braun and Robert Davis.
\newblock Ehrhart series, unimodality, and integrally closed reflexive
  polytopes.
\newblock {\em Ann. Comb.}, 20(4):705--717, 2016.

\bibitem{BraunDavisSolusIDP}
Benjamin Braun, Robert Davis, and Liam Solus.
\newblock Detecting the integer decomposition property and {E}hrhart
  unimodality in reflexive simplices.
\newblock {\em Advances in Applied Mathematics}, 100, 08 2016.

\bibitem{BraunLiu}
Benjamin Braun and Fu~Liu.
\newblock {$h^*$}-polynomials with roots on the unit circle.
\newblock {\em Experimental Mathematics}, pages 1--17, 2019.

\bibitem{BrunsRomer}
Winfried Bruns and Tim R{\"o}mer.
\newblock {$h$}-vectors of {G}orenstein polytopes.
\newblock {\em J. Combin. Theory Ser. A}, 114(1):65--76, 2007.

\bibitem{conrads}
Heinke Conrads.
\newblock Weighted projective spaces and reflexive simplices.
\newblock {\em Manuscripta Math.}, 107(2):215--227, 2002.

\bibitem{coxlittleoshea}
David~A. Cox, John Little, and Donal O'Shea.
\newblock {\em Ideals, Varieties, and Algorithms: An Introduction to
  Computational Algebraic Geometry and Commutative Algebra, 3/e (Undergraduate
  Texts in Mathematics)}.
\newblock Springer-Verlag, Berlin, Heidelberg, 2007.

\bibitem{hibiohsugi}
Takayuki Hibi and Hidefumi Ohsugi.
\newblock Quadratic initial ideals of root systems.
\newblock {\em Proc. Amer. Math. Soc.}, 130(7):1913--1922, 2001.

\bibitem{LiuSolusUnimodalityPositivity}
Fu~Liu and Liam Solus.
\newblock On the relationship between {E}hrhart unimodality and {E}hrhart
  positivity.
\newblock {\em Ann. Comb.}, 23(2):347--365, 2019.

\bibitem{NillSimplices}
Benjamin Nill.
\newblock Volume and lattice points of reflexive simplices.
\newblock {\em Discrete Comput. Geom.}, 37(2):301--320, 2007.

\bibitem{SolusNumeralSystems}
Liam Solus.
\newblock Simplices for numeral systems.
\newblock {\em Trans. Amer. Math. Soc.}, 371(3):2089--2107, 2019.

\bibitem{sturmfels}
Bernd Sturmfels.
\newblock {\em Gr\"obner {B}ases and {C}onvex {P}olytopes}, volume~8 of {\em
  University Lecture Series}.
\newblock American Mathematical Society, Providence, RI, 1996.

\end{thebibliography}

\end{document}